\documentclass[12pt,a4]{amsart}
\usepackage[a4paper, left=26mm, right=26mm, top=28mm, bottom=34mm]{geometry}
\usepackage[all]{xy}
\usepackage{amsmath}
\usepackage{amscd}
\usepackage{amssymb}
\usepackage{amsthm}
\usepackage{url}
\usepackage{mathrsfs}
\usepackage{comment}
\theoremstyle{definition}
\newtheorem{thm}{Theorem}[section]
\newtheorem{lem}[thm]{Lemma}
\newtheorem{cor}[thm]{Corollary}
\newtheorem{prop}[thm]{Proposition}
\newtheorem{conj}[thm]{Conjecture}
\theoremstyle{definition}
\newtheorem{ass}[thm]{Assumption}
\newtheorem{rem}[thm]{Remark}

\newtheorem{defn}[thm]{Definition}

\def\F{{\mathbb F}}
\def\G{{\mathbb G}}
\def\Q{{\mathbb Q}}

\def\Z{{\mathbb Z}}
\def\C{{\mathbb C}}
\def\O{{\mathcal O}}

\def\Br{\mathop{\mathrm{Br}}\nolimits}

\def\Gr{\text{\rm Gr}}

\def\cl{\text{\rm cl}}
\def\Aut{\mathop{\mathrm{Aut}}\nolimits}

\def\Coker{\mathop{\mathrm{Coker}}\nolimits}
\def\Frob{\mathop{\mathrm{Frob}}\nolimits}

\def\Gal{\mathop{\mathrm{Gal}}\nolimits}

\def\Hom{\mathop{\mathrm{Hom}}\nolimits}

\def\Im{\mathop{\mathrm{Im}}\nolimits}

\def\Ker{\mathop{\mathrm{Ker}}\nolimits}

\def\PGL{\mathop{\mathrm{PGL}}\nolimits}

\def\Pic{\mathop{\mathrm{Pic}}\nolimits}

\def\Spa{\mathop{\rm Spa}}
\def\Spec{\mathop{\rm Spec}}
\def\Spf{\mathop{\rm Spf}}

\def\sep{\text{\rm sep}}

\def\cris{\text{\rm cris}}

\def\rank{\mathop{\text{\rm rank}}\nolimits}

\def\chara{\mathop{\mathrm{char}}}

\newcommand{\plim}[1][]{\mathop{\varprojlim}\limits_{#1}}
\renewcommand{\labelenumi}{(\arabic{enumi})}
\def\Gr{\mathop{\mathrm{Gr}}\nolimits}

\newcommand{\et}{\mathrm{\acute{e}t}}

\newcommand{\ad}{\mathrm{ad}}
\newcommand{\ch}{\mathrm{ch}}
\def\NS{\mathop{\mathrm{NS}}\nolimits}

\def\WD{\mathop{\mathrm{WD}}\nolimits}
\def\CH{\mathop{\mathrm{CH}}\nolimits}

\begin{document}

\title[On a torsion analogue of the weight-monodromy conjecture]{On a torsion analogue of the weight-monodromy conjecture}

\author{Kazuhiro Ito}
\address{Department of Mathematics, Faculty of Science, Kyoto University, Kyoto 606-8502, Japan}
\email{kito@math.kyoto-u.ac.jp}




\maketitle

\begin{abstract}
We formulate and study a torsion analogue of the weight-monodromy conjecture for a proper smooth scheme over a non-archimedean local field.
We prove it for
proper smooth schemes over equal characteristic non-archimedean local fields,
abelian varieties,
surfaces,
varieties uniformized by Drinfeld upper half spaces,
and set-theoretic complete intersections in toric varieties.
In the equal characteristic case, our methods rely on an ultraproduct variant of Weil II established by Cadoret.
\end{abstract}

\section{Introduction} \label{Section:Introduction}

Let $K$ be a non-archimedean local field with finite residue field $k$.
Let $p >0$ be the characteristic of $k$,
and $q$ the number of elements in $k$.
Let $X$ be a proper smooth scheme over $K$ and $w$ an integer.
Let $\ell \neq p$ be a prime number.
Let $\overline{K}$ be an algebraic closure of $K$ and $K^\sep$ the separable closure of $K$ in $\overline{K}$.
The absolute Galois group
$G_K:=\Gal(K^{\sep}/K)$
naturally acts on
the $\ell$-adic cohomology
$H^{w}_\et(X_{\overline{K}}, \Q_\ell)$,
where we put $X_{\overline{K}}:=X \otimes_K \overline{K}$.

By Grothendieck's quasi-unipotence theorem,
the action of an open subgroup
of the inertia group $I_K$ of $K$
on
$H^{w}_\et(X_{\overline{K}}, \Q_\ell)$
defines the monodromy filtration
\[
\{ M_{i, \Q_\ell} \}_i
\]
on
$H^{w}_\et(X_{\overline{K}}, \Q_\ell)$.
It is an increasing filtration stable by the action of $G_K$.
(See Section \ref{Subsection:The weight-monodromy conjecture} for details.)
The \textit{weight-monodromy conjecture} due to Deligne states that the $i$-th graded piece
\[
\Gr^{M}_{i, \Q_\ell}:=M_{i, \Q_\ell}/M_{i-1, \Q_\ell}
\]
of the monodromy filtration on
$H^{w}_\et(X_{\overline{K}}, \Q_\ell)$
is of weight $w+i$,
i.e.\ every eigenvalue of a lift of the geometric Frobenius element $\Frob_k \in \Gal(\overline{k}/k)$ is an algebraic integer such that the complex absolute values of its conjugates are $q^{(w+i)/2}$.
When $X$ has good reduction over the ring of integers $\O_K$ of $K$, it is nothing more than the Weil conjecture \cite{WeilI, WeilII}.
In general, the weight-monodromy conjecture is still open.
In this paper, we shall propose a torsion analogue of the weight-monodromy conjecture and prove it in some cases.

By the work of Rapoport-Zink \cite{Rapoport-Zink} and de Jong's alteration \cite{deJong:Alteration},
we can take an open subgroup $J \subset I_K$ so that
the action of $J$ on
the \'etale cohomology group with $\F_\ell$-coefficients
$H^{w}_\et(X_{\overline{K}}, \F_\ell)$
is unipotent for every $\ell \neq p$.
By the same construction as in the $\ell$-adic case, we can define the monodromy filtration
\[
\{ M_{i, \F_\ell} \}_i
\]
on
$H^{w}_\et(X_{\overline{K}}, \F_\ell)$
for all but finitely many $\ell \neq p$, which is stable by the action of $G_K$;
see Section \ref{Subsection:A torsion analogue of the weight-monodromy conjecture} for details.
We propose the following conjecture.

\begin{conj}[A torsion analogue of the weight-monodromy conjecture, Conjecture \ref{Conjecture:A torsion analogue of the weight-monodromy conjecture}]\label{Conjecture:Torsion Introduction}
Let $X$ be a proper smooth scheme over $K$ and $w$ an integer.
For every $i$,
there exists a non-zero monic polynomial $P_i(T) \in \Z[T]$ satisfying the following conditions:
\begin{itemize}
\item
The roots of $P_i(T)$ have complex absolute values $q^{(w+i)/2}$.
\item
We have $P_i(\Frob)=0$ on the $i$-th graded piece
\[
\Gr^{M}_{i, \F_\ell}:=M_{i, \F_\ell}/M_{i-1, \F_\ell}
\]
for all but finitely many $\ell \neq p$ and for every lift
$\Frob \in G_K$
of the geometric Frobenius element.
\end{itemize}
\end{conj}

\begin{rem}\label{Remark:Gabber Weil conjecture}
The \'etale cohomology group with $\Z_\ell$-coefficients
$H^w_\et(X_{\overline{K}}, \Z_\ell)$ is torsion-free for all but finitely many $\ell \neq p$; see \cite{Gabber} and \cite[Theorem 1.4]{Suh}.
(See also \cite[Th\'eor\`eme 6.2.2]{Orgogozo}.)
When $X$ has good reduction over $\O_K$, Conjecture \ref{Conjecture:Torsion Introduction} follows from the Weil conjecture and this result.
\end{rem}

The main theorem of this paper is as follows:

\begin{thm}[Theorem \ref{Theorem:A torsion analogue of the weight-monodromy conjecture}]\label{Theorem:main}
Let $X$ be a proper smooth scheme over $K$ and $w$ an integer.
Assume that one of the following conditions holds:
\begin{enumerate}
    \item\label{equal} $K$ is of equal characteristic, i.e.\ the characteristic of $K$ is $p$.
    \item\label{abelian variety} $X$ is an abelian variety.
    \item\label{w=2} $w \leq 2$ or $w \geq 2 \dim X-2$.
    \item\label{Drinfeld} $X$ is uniformized by a Drinfeld upper half space.
    \item\label{complete intersection} $X$ is geometrically connected and is a set-theoretic complete intersection in a projective smooth toric variety.
\end{enumerate}
Then the assertion of Conjecture \ref{Conjecture:Torsion Introduction} for $(X, w)$ is true.
\end{thm}

The weight-monodromy conjecture for $\Q_\ell$-coefficients is known to be true for $(X, w)$ if one of the above conditions (\ref{equal})--(\ref{complete intersection}) holds for $(X, w)$.
However, it seems that the weight-monodromy conjecture for $\Q_\ell$-coefficients does not automatically imply
Conjecture \ref{Conjecture:Torsion Introduction}.
The problem is that, in general, we do not know the torsion-freeness of the cokernel of the monodromy operator acting on
$H^{w}_\et(X_{\overline{K}}, \Z_\ell)$ for all but finitely many $\ell \neq p$.
(See Section \ref{Subsection:torsion freeness of monodromy} for details.)

\begin{rem}\label{Remark:other cases intro}
There are other cases in which the weight-monodromy conjecture is known to be true; see Remark \ref{Remark:other cases}.
In this paper,
we will restrict ourselves to the cases (\ref{equal})--(\ref{complete intersection})
for the sake of simplicity.
\end{rem}

We shall give two applications of our results.
The first one is an application to the finiteness of the Brauer group of a proper smooth scheme over $K$ for which the $\ell$-adic Chern class map for divisors is surjective; see Corollary \ref{Corollary:finiteness Brauer group}.
As the second application,
we will show the finiteness of the $G_K$-fixed part of the prime-to-$p$ torsion part of the Chow group
$\CH^2(X_{\overline{K}})$ of codimension two cycles on
$X_{\overline{K}}$
if $(X, w=3)$ satisfies one of the conditions (\ref{equal})--(\ref{complete intersection}); see Corollary \ref{Corollary:finiteness of Chow group unconditional}.

The strategy of the proof of Theorem \ref{Theorem:main} is as follows.
If $K$ is of equal characteristic and $X$ is defined over the function field of a smooth curve over a finite field,
then Theorem \ref{Theorem:main} is a consequence of an ultraproduct variant of Weil II established by Cadoret \cite{Cadoret}.
The general case (\ref{equal}) can be deduced from this case by the same arguments as in \cite{Ito-equal char}.

As in \cite{Scholze},
by using the tilting equivalence of Scholze,
we will deduce the case (\ref{complete intersection}) from the case (\ref{equal}) or from the results of Cadoret \cite{Cadoret}.
In his proof of the weight-monodromy conjecture in the case (\ref{complete intersection}),
Scholze used a theorem of Huber \cite[Theorem 3.6]{Huber98b} on \'etale cohomology of tubular neighborhoods of rigid analytic varieties.
In our case,
we use a uniform variant \cite[Corollary 4.11]{Ito20} of Huber's theorem proved by the author.
See Section \ref{Section:Set-theoretic complete intersections cases in toric varieties} for details.

For the case (\ref{abelian variety}),
we prove that,
for abelian varieties,
the cokernels of the monodromy operators are torsion-free by using the theory of N\'eron models.
Then the case (\ref{abelian variety}) is deduced from the weight-monodromy conjecture for abelian varieties.
For the proof in the remaining cases,
we use the weight spectral sequence with $\Z_\ell$-coefficients constructed by Rapoport-Zink.
Since the weight-monodromy conjecture is known to be true under the assumptions,
it suffices to prove that the cokernels of the monodromy operators are torsion-free for all but finitely many $\ell \neq p$.
In our settings, it basically follows from the torsion-freeness for all but finitely many $\ell \neq p$ of the cokernel of a homomorphism of one of the following types:
\begin{itemize}
    \item The homomorphism $T_{\ell}A \to \Hom_{\Z_\ell}(T_{\ell}A, \Z_\ell(1))$ induced by a polarization of an abelian variety $A$ over $\overline{k}$. Here $T_\ell A$ is the $\ell$-adic Tate module of $A$.
    \item The base change $M_1\otimes_{\Z} \Z_\ell \to M_2  \otimes_{\Z} \Z_\ell$ of a homomorphism $M_1 \to M_2$ of finitely generated $\Z$-modules.
\end{itemize}

The outline of this paper is as follows.
In Section \ref{Section:Preliminaries},
we define a notion of weight for a family
$\{ H_\ell \}_{\ell \neq p}$
of $G_K$-representations over $\F_\ell$ and prepare some elementary lemmas used in this paper.
In Section \ref{Section:A torsion analogue of the weight-monodromy conjecture}, we define the monodromy filtration with coefficients in $\F_\ell$ for all but finitely many $\ell \neq  p$
and propose a torsion analogue of the weight-monodromy conjecture (Conjecture \ref{Conjecture:Torsion Introduction}).
We also discuss a relation between the weight-monodromy conjecture
and Conjecture \ref{Conjecture:Torsion Introduction}.
In Section \ref{Section:Torsion-freeness of the weight spectral sequence}, we discuss some torsion-freeness properties of the weight spectral sequence and their relation to Conjecture \ref{Conjecture:Torsion Introduction}.
In Section \ref{Section:Proof of the main theorem: equal and complete intersection}--\ref{Section:Proof of the main theorem: Drinfeld},
we prove Theorem \ref{Theorem:main}.
In Section \ref{Section:Some applications}, as applications of Theorem \ref{Theorem:main},
we discuss some finiteness properties of the Brauer group and
the codimension two Chow group of a proper smooth scheme over $K$.

\section{Preliminaries}\label{Section:Preliminaries}

\subsection{Weights}\label{Subsection:Weights}

Let $p$ be a prime number.
In this subsection,
we fix a finitely generated field $k$ over $\F_p$.
Let $\overline{k}$ be an algebraic closure of $k$ and $k^{\sep}$ the separable closure of $k$ in $\overline{k}$.
We put $G_k:=\Gal(k^\sep/k)$.
Let $\ell \neq p$ be a prime number.
We call a finitely generated $\Z_\ell$-module endowed with a continuous action of $G_k$ a
\textit{$G_k$-module over $\Z_\ell$}
for simplicity.
We will use the same notation for other fields throughout this paper.

Let $q$ be a power of $p$.
For a non-zero monic polynomial $P(T) \in \Z[T]$, we say that $P(T)$ is a \textit{Weil $q$-polynomial} if the complex absolute value of every root of $P(T)$ is $q^{1/2}$.

For a finite dimensional representation of
$G_k$
over $\Q_\ell$, there is a notion of weight; see \cite[Section 2.2]{Ito-equal char} for example.
In this paper, we will use the following notion of weight for a family
$\{ H_\ell \}_{\ell \neq p}$
of $G_k$-modules over $\Z_\ell$.
Let $\mathcal{L}$ be an infinite set of prime numbers $\ell \neq p$.
Let $w$ be an integer.

\begin{itemize}
    \item Let $U$ be an integral scheme of finite type over $\F_p$ with function field $k$.
We say that a family
$\{ \mathcal{F}_\ell \}_{\ell \in \mathcal{L}}$
of locally constant constructible $\Z_{\ell}$-sheaves on $U$ is \textit{of weight} $w$
if, for every closed point $x \in U$,
there is a Weil $(q_x)^w$-polynomial $P_x(T) \in \Z[T]$
such that,
for all but finitely many $\ell \in \mathcal{L}$,
we have
$P_x(\Frob_x)=0$ on
$\mathcal{F}_{\ell, \overline{x}}$.
Here $q_x$ is the number of elements in the residue field $\kappa(x)$ of $x$,
$\overline{x}$ is a geometric point of $U$ above $x$,
and
$\Frob_x \in G_{\kappa(x)}$, $a \mapsto a^{1/{q_x}}$
is the geometric Frobenius element.
    \item We say that a family $\{ H_{\ell} \}_{\ell \in \mathcal{L}}$ of $G_k$-modules over $\Z_\ell$ is \textit{of weight} $w$
if there is an integral scheme $U$ of finite type over $\F_p$ with function field $k$
such that the family $\{ H_{\ell} \}_{\ell \in \mathcal{L}}$
comes from a family
$\{ \mathcal{F}_\ell \}_{\ell \in \mathcal{L}}$
of locally constant constructible $\Z_{\ell}$-sheaves
on $U$ of weight $w$.
\end{itemize}
When there is no possibility of confusion, we will omit $\mathcal{L}$ from the notation and write $\{ H_{\ell} \}_{\ell}$ in place of $\{ H_{\ell} \}_{\ell \in \mathcal{L}}$.

\begin{lem}\label{Lemma:vanishing morphism}
Let $\{ H_{1, \ell} \}_{\ell \in \mathcal{L}}$ and $\{ H_{2, \ell} \}_{\ell \in \mathcal{L}}$ be families of $G_k$-modules over $\Z_\ell$ of weight $w_1$ and $w_2$, respectively.
We assume $w_1 \neq w_2$.
Then, for all but finitely many $\ell \in \mathcal{L}$,
every map
$H_{1, \ell} \to H_{2, \ell}$
of $G_k$-modules over $\Z_{\ell}$ is zero.
\end{lem}

\begin{proof}
We may assume that $\{ H_{1, \ell} \}_{\ell}$ and $\{ H_{2, \ell} \}_{\ell}$
come from families
$\{ \mathcal{F}_{1, \ell} \}_{\ell}$
and $\{ \mathcal{F}_{2, \ell} \}_{\ell}$
of locally constant constructible $\Z_{\ell}$-sheaves
on $U$ of weight $w_1$ and $w_2$, respectively.
Here $U$ is an integral scheme of finite type over $\F_p$ with function field $k$.
Take a closed point $x \in U$.
Let $P_{1, x}(T) \in \Z[T]$
be a Weil $(q_x)^{w_1}$-polynomial
such that,
for all but finitely many $\ell \in \mathcal{L}$,
we have
$P_{1, x}(\Frob_x)=0$
on
$(\mathcal{F}_{1, \ell})_{\overline{x}}$.
Let $P_{2, x}(T) \in \Z[T]$ be a Weil $(q_x)^{w_2}$-polynomial which satisfies the same condition for
$\{ \mathcal{F}_{2, \ell} \}_{\ell}$.
The polynomials $P_{1, x}(T)$ and $P_{2, x}(T)$ are relatively prime.
For $\ell \in \mathcal{L}$
such that $P_{1, x}(\Frob_x)=0$
on
$(\mathcal{F}_{1, \ell})_{\overline{x}}$
and
$P_{2, x}(\Frob_x)=0$
on
$(\mathcal{F}_{2, \ell})_{\overline{x}}$,
we have
$P_{1, x}(\Frob_x)=P_{2, x}(\Frob_x)=0$ on
the stalk of the image of any map
$\mathcal{F}_{1, \ell} \to \mathcal{F}_{2, \ell}$
at $\overline{x}$.
Therefore, the assertion follows from Lemma \ref{Lemma:relatively prime} below.
\end{proof}

\begin{lem}\label{Lemma:relatively prime}
Let $P_1(T), P_2(T) \in \Q[T]$ be two relatively prime polynomials.
For all but finitely many prime numbers $\ell$,
every $\Z_\ell[T]$-module $H_\ell$ such that
$P_1(T)=P_2(T)=0$ on
$H_\ell$
is zero.
\end{lem}

\begin{proof}
There exist polynomials $Q_1(T), Q_2(T) \in \Q[T]$ satisfying
\[
P_1(T)Q_1(T)+P_2(T)Q_2(T)=1
\]
in $\Q[T]$ since $P_1(T)$ and $P_2(T)$ are relatively prime.
Thus, for all but finitely many prime numbers $\ell$,
we have $P_1(T), P_2(T) \in \Z_\ell[T]$,
and they generate the unit ideal of $\Z_\ell[T]$.
The assertion follows from this fact.
\end{proof}

We need the following theorem
to define monodromy filtrations with coefficients in $\F_\ell$ for all but finitely many $\ell$
and to prove main results in this paper.

\begin{thm}[{Gabber \cite{Gabber}, Suh \cite[Theorem 1.4]{Suh}}]\label{Theorem:Gabber}
Let $X$ be a proper smooth scheme over a separably closed field of characteristic $p \geq 0$.
For all but finitely many $\ell \neq p$,
the $\Z_\ell$-module $H^w_\et(X, \Z_\ell)$ is torsion-free for every $w$.
In particular, for all but finitely many $\ell \neq p$,
the natural map
$H^w_\et(X, \Z_\ell) \to H^w_\et(X, \F_\ell)$
gives an isomorphism
\[
H^w_\et(X, \Z_\ell)\otimes_{\Z_\ell}\F_\ell \cong H^w_\et(X, \F_\ell)
\]
for every $w$.
\end{thm}

\begin{proof}
If $X$ is projective, this is a theorem of Gabber \cite[Theorem]{Gabber}.
(An alternative proof using ultraproduct Weil cohomology theory was obtained by Orgogozo; see \cite[Th\'eor\`eme 6.2.2]{Orgogozo}.)
By using de Jong's alteration \cite[Theorem 4.1]{deJong:Alteration},
the general case can be deduced from the projective case; see the proof of \cite[Theorem 1.4]{Suh} for details.
\end{proof}

\begin{cor}\label{Corollary:system of etale cohomology}
Let $X$ be a proper smooth scheme over $k$.
Then $\{ H^w_\et(X_{\overline{k}}, \Lambda_{\ell}) \}_{\ell \neq p}$ is a family of $G_k$-modules of weight $w$,
where $\Lambda_\ell$ is either $\Z_\ell$ or $\F_\ell$.
\end{cor}
\begin{proof}
This follows from the Weil conjecture \cite[Corollaire (3.3.9)]{WeilII} and Theorem \ref{Theorem:Gabber}.
\end{proof}

Let $K$ be a non-archimedean local field with finite residue field $k$.
Assume that the characteristic of $k$ is $p$.
Similarly, we will use the following notion of weight for
representations of $G_K$.
Let $q$ be the number of elements in $k$.
Let $I_K \subset G_K$ be the inertia group of $K$.
Let $\{ H_{\ell} \}_{\ell \in \mathcal{L}}$ be a family of $G_K$-modules over $\Z_{\ell}$.
We assume that there is an open subgroup $J \subset I_K$ such that the action of $J$ on $H_{\ell}$ is trivial for all but finitely many $\ell \in \mathcal{L}$.

\begin{defn}\label{Definition:weight}
We say that the family $\{ H_{\ell} \}_{\ell \in \mathcal{L}}$ is \textit{of weight} $w$
if there is a Weil $q^w$-polynomial $P(T) \in \Z[T]$
such that,
for all but finitely many $\ell \in \mathcal{L}$,
we have
$P(\Frob)=0$ on
$H_{\ell}$ for every lift $\Frob \in G_K$ of the geometric Frobenius element $\Frob_k \in G_k$, $a \mapsto a^{1/q}$.
\end{defn}

\begin{rem}\label{Remark:choice of Frob}
Since the action of the open subgroup $J \subset I_K$ on $H_{\ell}$ is trivial for all but finitely many $\ell \in \mathcal{L}$,
it follows that
the family $\{ H_{\ell} \}_{\ell \in \mathcal{L}}$ is of weight $w$ if and only if,
for \textit{one} lift $\Frob \in G_K$ of the geometric Frobenius element,
there is a Weil $q^w$-polynomial $P(T) \in \Z[T]$
such that,
for all but finitely many $\ell \in \mathcal{L}$,
we have $P(\Frob)=0$ on
$H_{\ell}$.
\end{rem}

\begin{lem}\label{Lemma:finite extension}
Let $L$ be a finite extension of $K$.
Then $\{ H_{\ell} \}_{\ell \in \mathcal{L}}$ is of weight $w$ as a family of $G_K$-modules over $\Z_{\ell}$ if and only if
$\{ H_{\ell} \}_{\ell \in \mathcal{L}}$ is of weight $w$ as a family of $G_L$-modules over $\Z_{\ell}$.
\end{lem}

\begin{proof}
Let $f$ be the residue degree of the extension $L/K$.
Let $\Frob \in G_{K}$ and
$\Frob' \in G_{L}$ be lifts of the geometric Frobenius elements.
There is a positive integer $n$
such that,
for all but finitely many $\ell \in \mathcal{L}$,
the action of $\Frob^{fn}$ on $H_\ell$ coincides with that of $(\Frob')^n$ on $H_\ell$.

We assume that $\{ H_{\ell} \}_{\ell \in \mathcal{L}}$ is of weight $w$ as a family of
$G_L$-modules.
Let $P(T) \in \Z[T]$ be a Weil $q^{fw}$-polynomial satisfying the condition in Definition \ref{Definition:weight}.
We write $P(T)$ in the form
$P(T)=\prod_{i}(T-\alpha_i)$ with $\alpha_i \in \overline{\Q}$.
We put $Q(T):=\prod_{i}(T^{fn}-\alpha^n_i) \in \Z[T]$, which is a Weil $q^{w}$-polynomial.
Then we have $Q(\Frob)=0$ on $H_\ell$ for all but finitely many $\ell \in \mathcal{L}$.
Therefore $\{ H_{\ell} \}_{\ell \in \mathcal{L}}$ is of weight $w$ as a family of $G_K$-modules.

The converse can be proved in a similar way.
\end{proof}

\subsection{Some elementary lemmas on nilpotent operators}\label{Subsection:Some elementary lemmas on nilpotent operators}

We collect some elementary lemmas on nilpotent operators,
which will be used in the sequel.

\begin{lem}\label{Lemma:torsion-freeness of cohomology groups}
Let $\ell$ be a prime number.
Let
$
M_1 \overset{f}{\longrightarrow} M_2 \overset{g}{\longrightarrow} M_3
$
be a complex of free $\Z_\ell$-modules of finite rank.
The reduction modulo $\ell$ of $f$ and $g$ will be denoted by
$\overline{f}$ and $\overline{g}$, respectively.
Hence we have the following complex of $\F_\ell$-vector spaces:
\[
M_1\otimes_{\Z_\ell}\F_\ell \overset{\overline{f}}{\longrightarrow} M_2\otimes_{\Z_\ell}\F_\ell \overset{\overline{g}}{\longrightarrow} M_3\otimes_{\Z_\ell}\F_\ell.
\]
Then we have
    \[
    \rank_{\Z_\ell}({\Ker{g}/\Im{f}}) \leq \dim_{\F_\ell}({\Ker{\overline{g}}/\Im{\overline{f}}}).
    \]
The equality holds if and only if the $\Z_\ell$-modules $\Coker{f}$ and $\Coker{g}$ are torsion-free.
If this is the case, then we have
$(\Ker{g}/\Im{f})\otimes_{\Z_\ell}\F_\ell \cong {\Ker{\overline{g}}/\Im{\overline{f}}}$.
\end{lem}

\begin{proof}
By the theory of elementary divisors, we have
\[
\rank_{\Z_\ell}({\Ker{g}/\Im{f}}) \leq \dim_{\F_\ell}{(\Ker{g}/\Im{f})\otimes_{\Z_\ell}\F_\ell},
\]
and the equality holds if and only if $\Ker{g}/\Im{f}$ is torsion-free.
Since $M_3$ is torsion-free, we see that $\Ker{g}/\Im{f}$ is torsion-free if and only if $\Coker{f}$ is torsion-free.
Moreover, we have inclusions
$
\Im{\overline{f}} \subset (\Ker{g})\otimes_{\Z_\ell}\F_\ell \subset \Ker{\overline{g}}.
$
Hence we have
\[
\dim_{\F_\ell}{(\Ker{g}/\Im{f})\otimes_{\Z_\ell}\F_\ell} \leq \dim_{\F_\ell}{(\Ker{\overline{g}}/\Im{\overline{f}})},
\]
and the equality holds if and only if $(\Ker{g})\otimes_{\Z_\ell}\F_\ell = \Ker{\overline{g}}$.
It is easy to see that
$(\Ker{g})\otimes_{\Z_\ell}\F_\ell = \Ker{\overline{g}}$
if and only if
$\Coker{g}$
is torsion-free.
This fact completes the proof of the lemma.
\end{proof}

\begin{lem}\label{Lemma:unipotent nilpotent}
Let $V$ be a vector space of dimension $n$ over a field $\F$ of positive characteristic $\ell$.
We assume that $\ell \geq n$.
For a unipotent operator $U$ on $V$, we define
\[
\log(U):= \sum_{1 \leq i \leq n-1} \frac{(-1)^{i+1}}{i} (U-1)^i.
\]
For a nilpotent operator $N$ on $V$, we define
\[
\exp(N):= \sum_{0 \leq i \leq n-1} \frac{1}{i!} N^i.
\]
Then the following assertions hold.
\begin{enumerate}
\renewcommand{\labelenumi}{(\roman{enumi})}
    \item $\log(-)$ defines a bijection from the set of
    unipotent operators on $V$ to the set of
    nilpotent operators on $V$ with inverse map $\exp(-)$.
    \item For two unipotent operators $U, U'$ (resp.\ two nilpotent operators $N, N'$) such that they commute, we have
    $\log(UU')=\log(U)+\log(U')$ (resp.\ $\exp(N+N')=\exp(N)\exp(N')$).
\end{enumerate}
\end{lem}

\begin{proof}
Although this lemma is well known, we recall the proof for the reader's convenience.

(i) Let $\Z_{(\ell)}$ be the localization of $\Z$ at the prime ideal $(\ell)$.
It suffices to prove that the homomorphism
$
\Z_{(\ell)}[S]/(S-1)^n \to \Z_{(\ell)}[T]/(T)^n
$,
$
S\mapsto \exp(T)
$
and the homomorphism
$
\Z_{(\ell)}[T]/(T)^n \to \Z_{(\ell)}[S]/(S-1)^n
$,
$
T \mapsto \log(S)
$
are inverse to each other, where $\exp(-)$ and $\log(-)$ are defined by the same formulas as above.
Since both rings are torsion-free over $\Z_{(\ell)}$, it suffices to prove the claim after tensoring with $\Q$.
Then it follows from the fact that
the map
$\Q[[S-1]] \to \Q[[T]]$, $S-1 \mapsto \exp(T)-1$
and
the map
$\Q[[T]] \to \Q[[S-1]]$, $T \mapsto \log(S)$
are inverse to each other, where $\exp(-)$ and $\log(-)$ are defined in the usual way.

(ii) By (i), we only need to prove that, for two nilpotent operators $N, N'$ such that they commute, we have $\exp(N+N')=\exp(N)\exp(N')$.
We have $N^i(N')^j=0$ on $V$ for $i, j \geq 0$ with $i+j \geq n$.
Thus, it suffices to prove $\exp(T+T')=\exp(T)\exp(T')$ in $\Z_{(\ell)}[T, T']/(T, T')^n$,
where $\exp(-)$ is defined by the same formula as above.
As in (i), this can be deduced from an analogous statement for $\Q[[T, T']]$.
\end{proof}

Let $R$ be a principal ideal domain and
$F$ its field of fractions.
Let $H$ be a free $R$-module of finite rank.
Let $N \colon H \to H$ be a nilpotent homomorphism.
By \cite[Proposition (1.6.1)]{WeilII},
the nilpotent homomorphism
$N_F:=N \otimes_R F$
on
$H_F:=H\otimes_R F$
determines
a unique increasing, separated, exhaustive filtration $\{ M_{i, F} \}_i$ on
$H_F$
characterized by the following properties:
\begin{itemize}
    \item $N_F(M_{i, F}) \subset M_{i-2, F}$ for every $i$.
    \item For every integer $i \geq 0$, the $i$-th iterate
    $N^i_F$ induces an isomorphism
    $
    \Gr^{M}_{i, F} \cong \Gr^{M}_{-i, F}.
    $
    Here we put
$\Gr^{M}_{i, F}:=M_{i, F}/M_{i-1, F}$.
\end{itemize}
We call $\{ M_{i, F} \}_i$ the filtration on $H_F$ associated with $N_F$.
Let $\{ M_{i} \}_i$ be a filtration on the $R$-module $H$ defined by
$
M_i := H \cap M_{i, F}
$
for every $i$.
The $R$-module
$\Gr^{M}_i:=M_{i}/M_{i-1}$
is torsion-free for every $i$.

\begin{lem}\label{Lemma:filtration and cokernel}
Let the notation be as above.
The cokernel of the $i$-th iterate $N^i \colon H \to H$ of $N$ is torsion-free for every $i \geq 0$ if and only if
$N^i$ induces an isomorphism
$
\Gr^{M}_i \cong \Gr^{M}_{-i}
$
for every $i \geq 0$.
\end{lem}
\begin{proof}
Assume that the cokernel of $N^i \colon H \to H$ is torsion-free for every $i \geq 0$.
Let $d \geq 0$ be the smallest integer such that $N^{d+1}=0$.
The cokernel of the $i$-th iterate of the homomorphism
\[
\Ker N^d/\Im N^d \to \Ker N^d/\Im N^d
\]
induced by $N$ is torsion-free for every $i \geq 0$.
Thus, by the same argument as in the proof of \cite[Proposition (1.6.1)]{WeilII}, we can construct inductively an increasing, separated, exhaustive filtration
$\{ M'_i \}_i$ on $H$
satisfying the following properties:
\begin{itemize}
    \item $\Gr^{M'}_i:=M'_{i}/M'_{i-1}$ is torsion-free for every $i$.
    \item $N(M'_{i}) \subset M'_{i-2}$ for every $i$.
    \item For every integer $i \geq 0$, the $i$-th iterate
    $N^i$ induces an isomorphism
    $
    \Gr^{M'}_i \cong \Gr^{M'}_{-i}.
    $
\end{itemize}
By uniqueness, the filtration
$\{ M_{i, F} \}_i$ coincides with
$\{ M'_{i}\otimes_R F \}_i$.
Since both $\Gr^{M}_i$ and $\Gr^{M'}_i$ are torsion-free for every $i$,
the filtration $\{ M_{i} \}_i$ coincides with $\{ M'_{i} \}_i$ and we have an isomorphism
$
N^i \colon \Gr^{M}_i \cong \Gr^{M}_{-i}
$
for every $i \geq 0$.

Conversely, we assume that $N^i$ induces an isomorphism
$
\Gr^{M}_i \cong \Gr^{M}_{-i}
$
for every $i \geq 0$.
We fix an integer $i \geq 0$.
For every $j \leq i$,
the $i$-th iterate
$N^i \colon \Gr^{M}_j \to \Gr^{M}_{j-2i}$ is surjective.
It follows that
$
N^i \colon M_j \to M_{j-2i}
$
is surjective for every $j \leq i$ since it is surjective for sufficiently small $j$.
For every $j \geq i$,
the $i$-th iterate
$N^i \colon \Gr^{M}_j \to \Gr^{M}_{j-2i}$ is a split injection.
It follows that the cokernel of
$
N^i \colon M_j \to M_{j-2i}
$
is torsion-free for every $j \geq i$ since we have shown that it is zero for
$j=i$.
Hence the cokernel of $N^i \colon H \to H$ is torsion-free.
\end{proof}

\section{A torsion analogue of the weight-monodromy conjecture}\label{Section:A torsion analogue of the weight-monodromy conjecture}

In this section, let $K$ be a non-archimedean local field with ring of integers $\O_K$.
Let $k$ be the finite residue field of $\O_K$.
Let $p >0$ be the characteristic of $k$,
and $q$ the number of elements in $k$.
Let $I_K \subset G_K$ be the inertia group of $K$.
For a prime number $\ell \neq p$,
the group of $\ell^n$-th roots of unity in $\overline{K}$
is denoted by $\mu_{\ell^n}$.
Let
\[
t_{\ell} \colon I_K \to \Z_\ell(1):= \plim[n]\mu_{\ell^n}
\]
be the map defined by
$g \mapsto \{ g(\varpi^{1/\ell^n})/\varpi^{1/\ell^n} \}_n$
for a uniformizer $\varpi \in \O_K$.
This map is independent of the choice of $\varpi$ and gives the maximal pro-$\ell$ quotient of $I_K$.
Let $X$ be a proper smooth scheme over $K$ and $w$ an integer.

\subsection{The weight-monodromy conjecture}\label{Subsection:The weight-monodromy conjecture}

We shall recall the definition of the monodromy filtration on
$H^{w}_\et(X_{\overline{K}}, \Q_\ell)$
for every $\ell \neq p$, where
$X_{\overline{K}}:=X \otimes_K \overline{K}$.
The absolute Galois group
$G_K$
naturally acts on
$H^{w}_\et(X_{\overline{K}}, \Q_\ell)$
via the natural isomorphism
$\Aut(\overline{K}/K) \cong G_K$.

By Grothendieck's quasi-unipotence theorem,
there is an open subgroup $J$ of $I_K$ such that the action of $J$ on $H^{w}_\et(X_{\overline{K}}, \Q_\ell)$ is unipotent
and factors through
$
t_{\ell}.
$
Take an element $\sigma \in J$ such that $t_{\ell}(\sigma) \in \Z_\ell(1)$ is a non-zero element.
We define
\[
N_{\sigma}:=\log(\sigma):= \sum_{1\leq i} \frac{(-1)^{i+1}}{i} (\sigma-1)^i \colon H^{w}_\et(X_{\overline{K}}, \Q_\ell) \to H^{w}_\et(X_{\overline{K}}, \Q_\ell).
\]
Let $\{ M_{i, \Q_\ell} \}_i$ be the filtration on $H^{w}_\et(X_{\overline{K}}, \Q_\ell)$ associated with
$N_\sigma$; see \cite[Proposition (1.6.1)]{WeilII}.
The filtration $\{ M_{i, \Q_\ell} \}_i$ is independent of $J$ and $\sigma \in J$.
It is called the \textit{monodromy filtration}.
We have
$\chi_\mathrm{cyc}(g) N_\sigma g=g N_\sigma$
for every $g \in G_K$,
where
$\chi_\mathrm{cyc} \colon G_K \to \Z^{\times}_\ell$
is the $\ell$-adic cyclotomic character.
It follows from the uniqueness of the monodromy filtration
that $\{ M_{i, \Q_\ell} \}_i$ is stable by the action of $G_K$.
We note that the filtration associated with
$\sigma-1$
coincides with $\{ M_{i, \Q_\ell} \}_i$.
We put
\[
\Gr^{M}_{i, \Q_\ell}:=M_{i, \Q_\ell}/M_{i-1, \Q_\ell}.
\]

We recall the weight-monodromy conjecture due to Deligne.

\begin{conj}[Deligne \cite{HodgeI}]\label{Conjecture:Weight-monodromy}
Let $X$ be a proper smooth scheme over $K$ and $w$ an integer.
Let $\ell \neq p$ be a prime number.
Then the $i$-th graded piece
$\Gr^{M}_{i, \Q_\ell}$
of the monodromy filtration on
$H^{w}_\et(X_{\overline{K}}, \Q_\ell)$ is of weight $w+i$,
i.e.\ every eigenvalue of every lift
$\Frob \in G_K$
of the geometric Frobenius element
is an algebraic integer such that the complex absolute values of its conjugates are $q^{(w+i)/2}$.
\end{conj}
When $X$ has good reduction over $\O_K$, it is nothing more than the Weil conjecture.
Conjecture \ref{Conjecture:Weight-monodromy} is known to be true in the following cases.

\begin{thm}\label{Theorem:Weight-monodromy conjecture}
Conjecture \ref{Conjecture:Weight-monodromy} for $(X, w)$ is true in the following cases:
\begin{enumerate}
    \item $K$ is of equal characteristic (\cite{WeilII, Terasoma, Ito-equal char}).
    \item $X$ is an abelian variety (\cite[Expos\'e IX]{SGA 7-I}).
    \item $w \leq 2$ or $w \geq 2\dim X-2$ (\cite{Rapoport-Zink, deJong:Alteration, Saito}).
    \item $X$ is uniformized by a Drinfeld upper half space (\cite{Ito-p-adic uniformized, Dat}).
    \item $X$ is geometrically connected and is a set-theoretic complete intersection in a projective smooth toric variety (\cite{Scholze}).
\end{enumerate}
\end{thm}

\begin{proof}
See the references given above.
\end{proof}

We will prove a torsion analogue of Conjecture \ref{Conjecture:Weight-monodromy} in each of the above cases.

\begin{rem}\label{Remark:other cases}
There are other cases in which Conjecture \ref{Conjecture:Weight-monodromy} is known to be true.
For example, in \cite{Ito-three fold},
it is proved for a certain projective threefold with strictly semistable reduction,
and in \cite{Mieda},
it is proved for a variety which is uniformized by a product of Drinfeld upper half spaces.
We will not discuss a torsion analogue of Conjecture \ref{Conjecture:Weight-monodromy} for these varieties in this paper for the sake of simplicity.
\end{rem}

\subsection{A torsion analogue of the weight-monodromy conjecture}\label{Subsection:A torsion analogue of the weight-monodromy conjecture}

Let
$\{ H_{\ell} \}_{\ell \neq p}$
be a family of finite dimensional $G_K$-representations over $\F_{\ell}$.
We define the monodromy filtrations when
the family
$\{ H_{\ell} \}_{\ell \neq p}$
satisfies the following two conditions:
\begin{itemize}
    \item There is an open subgroup $J$ of $I_K$
    such that, for every $\ell \neq p$,
    the action of $J$ on $H_\ell$ is unipotent (i.e.\ $\sigma$ is a unipotent operator on $H_\ell$ for every
$\sigma \in J$).
    \item $n:=\sup_{\ell \neq p} \dim_{\F_\ell} H_\ell < \infty$.
\end{itemize}
The action of $J$ factors through
$
t_{\ell}
$
for every $\ell \neq p$.
Take an element $\sigma \in J$
such that,
for all but finitely many $\ell \neq p$,
the image $t_{\ell}(\sigma) \in \Z_\ell(1)$
is a generator.
For a prime number $\ell \neq p$ with $\ell \geq n$,
we define
\[
N_\sigma := \log(\sigma):=\sum_{1 \leq i \leq n-1} \frac{(-1)^{i+1}}{i} (\sigma-1)^i
\colon H_\ell \to H_\ell.
\]
(See also Lemma \ref{Lemma:unipotent nilpotent}.)
Let
\[
\{ M_{i, \F_\ell} \}_i
\]
be the filtration on $H_\ell$ associated with $N_\sigma$.
The filtration $\{ M_{i, \F_\ell} \}_i$ is independent of $J$ and $\sigma \in J$ up to excluding finitely many $\ell \neq p$.
Moreover, for all but finite many $\ell \neq p$,
we have
\[
\overline{\chi_\mathrm{cyc}(g)} N_\sigma g=g N_\sigma
\]
for every $g \in G_K$,
where $\overline{\chi_\mathrm{cyc}(g)}$
is the reduction modulo $\ell$ of
$\chi_\mathrm{cyc}(g)$, and
$\{ M_{i, \F_\ell} \}_i$ is stable by the action of $G_K$.
We note that the filtration induced by
$\sigma-1$
coincides with $\{ M_{i, \F_\ell} \}_i$ up to excluding finitely many $\ell \neq p$.
We call $\{ M_{i, \F_\ell} \}_i$ the
\textit{monodromy filtration with coefficients in} $\F_\ell$ on $H_\ell$.
For all but finitely many $\ell \neq p$,
the action of $J$ is trivial on $M_{i, \F_\ell}/M_{i-1, \F_\ell}$ for every $i$, and we can ask whether
the family
$\{ M_{i, \F_\ell}/M_{i-1, \F_\ell} \}_\ell$
of $G_K$-representations over $\F_\ell$
is of weight $w$ for some integer $w$ in the sense of Definition \ref{Definition:weight}.

Now let us come back to our original setting.
By the work of Rapoport-Zink \cite{Rapoport-Zink} and de Jong's alteration \cite[Theorem 6.5]{deJong:Alteration},
there is an open subgroup $J$ of $I_K$
such that,
for every $\ell \neq p$,
the action of $J$ on $H^{w}_\et(X_{\overline{K}}, \Lambda_\ell)$ is unipotent
and factors through
$
t_{\ell},
$
where $\Lambda_\ell$ is $\Q_\ell$, $\Z_\ell$, or $\F_\ell$.
(See also \cite[Proposition 6.3.2]{Berthelot}.)
By Theorem \ref{Theorem:Gabber}, we have
\[
\sup_{\ell \neq p}\dim_{\F_\ell} H^{w}_\et(X_{\overline{K}}, \F_\ell) < \infty.
\]
(Alternatively,
this fact can be proved by using the argument in \cite[Section 6.2.4]{Orgogozo}.)
Therefore,
the family
$\{ H^{w}_\et(X_{\overline{K}}, \F_\ell) \}_{\ell \neq p}$
satisfies the above two conditions,
and
we have the monodromy filtration
$
\{ M_{i, \F_\ell} \}_i
$
with coefficients in $\F_\ell$ on $H^{w}_\et(X_{\overline{K}}, \F_\ell)$
for all but finitely many $\ell \neq p$.
We put
\[
\Gr^{M}_{i, \F_\ell}:=M_{i, \F_\ell}/M_{i-1, \F_\ell}.
\]
Here we omit $X$ and $w$ from the notation.
This will not cause any confusion in the context.

A torsion analogue of Conjecture \ref{Conjecture:Weight-monodromy} can be formulated as follows:

\begin{conj}\label{Conjecture:A torsion analogue of the weight-monodromy conjecture}
Let $X$ be a proper smooth scheme over $K$ and $w$ an integer.
The family
$
\{ \Gr^{M}_{i, \F_\ell} \}_{\ell}
$
of finite dimensional $G_K$-representations over $\F_\ell$ defined above is of weight $w+i$ for every $i$ in the sense of Definition \ref{Definition:weight}.
\end{conj}

\begin{rem}\label{Remark:finite extension}\
\begin{enumerate}
\renewcommand{\labelenumi}{(\roman{enumi})}
    \item For a finite extension $L$ of $K$,
Conjecture \ref{Conjecture:A torsion analogue of the weight-monodromy conjecture} for $(X, w)$ is equivalent to
Conjecture \ref{Conjecture:A torsion analogue of the weight-monodromy conjecture} for $(X_{L}, w)$ by Lemma \ref{Lemma:finite extension}.
    \item When $X$ has good reduction over $\O_K$,
Conjecture \ref{Conjecture:A torsion analogue of the weight-monodromy conjecture}
is a consequence of the Weil conjecture and Theorem \ref{Theorem:Gabber}; see Corollary \ref{Corollary:system of etale cohomology}.
\end{enumerate}
\end{rem}

The main theorem of this paper is as follows.

\begin{thm}\label{Theorem:A torsion analogue of the weight-monodromy conjecture}
Let $X$ be a proper smooth scheme over $K$ and $w$ an integer.
We assume that $(X, w)$ satisfies one of the conditions (\ref{equal})--(\ref{complete intersection}) in Theorem \ref{Theorem:Weight-monodromy conjecture}.
Then the assertion of Conjecture \ref{Conjecture:A torsion analogue of the weight-monodromy conjecture} for $(X, w)$ is true.
\end{thm}

We will prove Theorem \ref{Theorem:A torsion analogue of the weight-monodromy conjecture} in Section \ref{Section:Proof of the main theorem: equal and complete intersection}--\ref{Section:Proof of the main theorem: Drinfeld}.

\subsection{Torsion-freeness of monodromy operators}\label{Subsection:torsion freeness of monodromy}

In this subsection,
we discuss a relation between
Conjecture \ref{Conjecture:Weight-monodromy}
and
Conjecture \ref{Conjecture:A torsion analogue of the weight-monodromy conjecture}.

Let $J$ be an open subgroup of $I_K$ such that the action of $J$ on $H^{w}_\et(X_{\overline{K}}, \Lambda_\ell)$ is unipotent for every $\ell \neq p$, where $\Lambda_\ell$ is $\Q_\ell$, $\Z_\ell$, or $\F_\ell$.
Take an element $\sigma \in J$ such that $t_{\ell}(\sigma) \in \Z_\ell(1)$ is
a generator for all but finitely many $\ell \neq p$.

\begin{lem}\label{Lemma:property (t-f)}
By pulling back the monodromy filtration $\{ M_{i, \Q_\ell} \}_i$ on $H^{w}_\et(X_{\overline{K}}, \Q_\ell)$, we define a filtration $\{ M_{i, \Z_\ell} \}_i$ on $H^{w}_\et(X_{\overline{K}}, \Z_\ell)$.
Then the following two statements for $(X, w)$ are equivalent:
\begin{enumerate}
\renewcommand{\labelenumi}{(\roman{enumi})}
    \item For all but finitely many $\ell \neq p$, the reduction modulo $\ell$ of $\{ M_{i, \Z_\ell} \}_i$ coincides with the monodromy filtration
$\{ M_{i, \F_\ell} \}_i$
with coefficients in $\F_\ell$ via the isomorphism
$H^{w}_\et(X_{\overline{K}}, \Z_\ell)\otimes_{\Z_\ell}\F_\ell \cong H^{w}_\et(X_{\overline{K}}, \F_\ell)$ (see Theorem \ref{Theorem:Gabber}).
    \item The cokernel of
\[
(\sigma-1)^i \colon H^{w}_\et(X_{\overline{K}}, \Z_\ell) \to H^{w}_\et(X_{\overline{K}}, \Z_\ell)
\]
is torsion-free for all but finitely many $\ell \neq p$ and every $i \geq 0$.
\end{enumerate}
\end{lem}

\begin{proof}
Use Theorem \ref{Theorem:Gabber}, Lemma \ref{Lemma:filtration and cokernel} and Nakayama's lemma.
\end{proof}

\begin{defn}\label{Definition:property (t-f)}
If the two equivalent statements in Lemma \ref{Lemma:property (t-f)} hold for $(X, w)$, then we say that $(X, w)$ \textit{satisfies the property} (t-f).
\end{defn}

Let $G$ be a group and let $M$ be an abelian group equipped with an action of $G$.
Let $M^{G}$ denote the $G$-fixed part of $M$.
Let $M_G$ denote be the group of $G$-coinvariants of $M$.

\begin{prop}\label{Proposition:cokernel of monodromy operator}\
\begin{enumerate}
\renewcommand{\labelenumi}{(\roman{enumi})}
    \item If $(X, w)$ satisfies the property (t-f),
    then for all but finitely many $\ell \neq p$,
    we have
    \[
H^{w}_\et(X_{\overline{K}}, \Z_\ell)^{I_K}\otimes_{\Z_\ell}\F_\ell \cong H^{w}_\et(X_{\overline{K}}, \F_\ell)^{I_K}
\]
and
the $\Z_\ell$-module $H^{w}_\et(X_{\overline{K}}, \Z_\ell)_{I_K}$
is torsion-free.
    \item If Conjecture \ref{Conjecture:A torsion analogue of the weight-monodromy conjecture} for $(X, w)$ is true, then $(X, w)$ satisfies the property (t-f).
    \item Assume that Conjecture \ref{Conjecture:Weight-monodromy} for $(X, w)$ is true and $(X, w)$ satisfies the property (t-f).
    Then Conjecture \ref{Conjecture:A torsion analogue of the weight-monodromy conjecture} for $(X, w)$ is true.
\end{enumerate}
\end{prop}

\begin{proof}
In the proof,
we will use
the weight filtration
$\{ W_{i, \Q_\ell} \}_i$ on
$H^{w}_\et(X_{\overline{K}}, \Q_\ell)$,
which we will recall in Remark \ref{Remark:weight filtration}.

(i) We may assume that $J=I_K$.
Then the assertion follows from the torsion-freeness of the cokernel of
$
\sigma-1 \colon H^{w}_\et(X_{\overline{K}}, \Z_\ell) \to H^{w}_\et(X_{\overline{K}}, \Z_\ell)
$
and Lemma \ref{Lemma:torsion-freeness of cohomology groups}.

(ii) Let $\{ W_{i, \Q_\ell} \}_i$ be the weight filtration on
$H^{w}_\et(X_{\overline{K}}, \Q_\ell)$.
By pulling buck $\{ W_{i, \Q_\ell} \}_i$ to $H^{w}_\et(X_{\overline{K}}, \Z_\ell)$,
we have a filtration
$\{ W_{i, \Z_\ell} \}_i$
on
$H^{w}_\et(X_{\overline{K}}, \Z_\ell)$.
We have
$(\sigma-1)(W_{i, \Z_\ell}) \subset W_{i-2, \Z_\ell}$
and
the $i$-th graded piece
$\Gr^{W}_{i, \Z_\ell}:=W_{i, \Z_\ell}/W_{i-1, \Z_\ell}$
is torsion-free for every $i$.
By Theorem \ref{Theorem:Gabber},
for all but finitely many $\ell \neq p$,
we can define
a filtration $\{ W_{i, \F_\ell} \}_i$ on $H^{w}_\et(X_{\overline{K}}, \F_\ell)$ by taking the reduction modulo $\ell$ of the  filtration
$\{ W_{i, \Z_\ell} \}_i$.
We define
$\Gr^{W}_{i, \F_\ell}:=W_{i, \F_\ell}/W_{i-1, \F_\ell}$.
Then the family $\{ \Gr^{W}_{i, \F_\ell} \}_{\ell}$ is of weight $w+i$;
see Proposition \ref{Proposition:polynomial Frob vanishing} (ii) in Section \ref{Section:Torsion-freeness of the weight spectral sequence}.

Now we assume that Conjecture \ref{Conjecture:A torsion analogue of the weight-monodromy conjecture} for $(X, w)$ is true.
Then $\{ W_{i, \F_\ell} \}_i$ coincides with the monodromy filtration $\{ M_{i, \F_\ell} \}_i$ with coefficients in $\F_\ell$ for all but finitely many $\ell \neq p$ by Lemma \ref{Lemma:relatively prime}.
Thus, the $i$-th iterate $(\sigma-1)^i$ of $\sigma-1$ induces an isomorphism
\[
(\sigma-1)^i \colon \Gr^{W}_{i, \F_\ell} \cong \Gr^{W}_{-i, \F_\ell}
\]
for every $i \geq 0$ and all but finitely many $\ell \neq p$.
By Nakayama's lemma, we have
\[
(\sigma-1)^i \colon \Gr^{W}_{i, \Z_\ell} \cong \Gr^{W}_{-i, \Z_\ell}
\]
for every $i \geq 0$ and all but finitely many $\ell \neq p$.
It follows that the weight filtration
$\{ W_{i, \Q_\ell} \}_i$
coincides with the monodromy filtration $\{ M_{i, \Q_\ell} \}_i$ for all but finitely many $\ell \neq p$, and the condition (i) in Lemma \ref{Lemma:property (t-f)} is satisfied.

(iii) Assume that Conjecture \ref{Conjecture:Weight-monodromy} for $(X, w)$ is true.
Then the weight filtration
$\{ W_{i, \Q_\ell} \}_i$
coincides with the monodromy filtration $\{ M_{i, \Q_\ell} \}_i$ for every $\ell \neq p$.
Assume further that $(X, w)$ satisfies the property (t-f).
Then it follows that the monodromy filtration $\{ M_{i, \F_\ell} \}_i$ coincides with
$\{ W_{i, \F_\ell} \}_i$ for all but finitely many $\ell \neq p$.
Thus, Conjecture \ref{Conjecture:A torsion analogue of the weight-monodromy conjecture} for $(X, w)$ is true.
\end{proof}

For later use, we state the following result as a corollary.

\begin{cor}\label{Corollary:cokernel of monodromy operator}
Assume that $(X, w)$ satisfies one of conditions (\ref{equal})--(\ref{complete intersection}) in Theorem \ref{Theorem:Weight-monodromy conjecture}.
Then,
for all but finitely many $\ell \neq p$,
we have
$
H^{w}_\et(X_{\overline{K}}, \Z_\ell)^{I_K}\otimes_{\Z_\ell}\F_\ell \cong H^{w}_\et(X_{\overline{K}}, \F_\ell)^{I_K}
$
and
the $\Z_\ell$-module $H^{w}_\et(X_{\overline{K}}, \Z_\ell)_{I_K}$
is torsion-free.
\end{cor}
\begin{proof}
Use Theorem \ref{Theorem:Weight-monodromy conjecture}, Theorem \ref{Theorem:A torsion analogue of the weight-monodromy conjecture}, and Proposition \ref{Proposition:cokernel of monodromy operator}.
\end{proof}

\begin{rem}\label{Remark:CHT}
Let $Z$ be a proper smooth scheme over a finitely generated field $F$ over $\F_p$.
Cadoret-Hui-Tamagawa proved that the natural map
$H^{w}_\et(Z_{\overline{F}}, \Z_\ell) \to H^{w}_\et(Z_{\overline{F}}, \F_\ell)$
gives an isomorphism
\[
H^{w}_\et(Z_{\overline{F}}, \Z_\ell)^{\Gal(F^\sep/{F. \overline{\F}_p})}\otimes_{\Z_\ell}\F_\ell \cong H^{w}_\et(Z_{\overline{F}}, \F_\ell)^{\Gal(F^\sep/{F.\overline{\F}_p})}
\]
for all but finitely many $\ell \neq p$;
see \cite[Theorem 4.5]{CHT2}.
Corollary \ref{Corollary:cokernel of monodromy operator} is a local analogue of this result.
\end{rem}

\section{Torsion-freeness of the weight spectral sequence}\label{Section:Torsion-freeness of the weight spectral sequence}
Let $K$ be a Henselian discrete valuation field with ring of integers $\O_K$.
The residue field of $\O_K$ is denoted by $k$.
Let $p \geq 0$ be the characteristic of $k$.
For a prime number $\ell \neq p$, let
$t_{\ell} \colon I_K \to \Z_\ell(1)$
be a map defined in the same way as in Section \ref{Section:A torsion analogue of the weight-monodromy conjecture}.
Let $\varpi \in \O_K$ be a uniformizer.

Let $\mathfrak{X}$ be a proper scheme over $\O_K$.
We assume that $\mathfrak{X}$ is \textit{strictly semi-stable} over $\O_K$ purely of relative dimension $d$, i.e.\ it is, Zariski locally on $\mathfrak{X}$, \'etale over
\[
\Spec{\O_K[T_0, \dotsc , T_d]/(T_0 \cdots T_r-\varpi)}
\]
for an integer $r$ with $0 \leq r \leq d$.

Let $X$ and $Y$ be the generic fiber and the special fiber of $\mathfrak{X}$, respectively.
Let $D_1, \dotsc, D_m$ be the irreducible components of $Y$.
We equip each $D_i$ with the reduced induced subscheme structure.
Following \cite{Saito}, we introduce some notation.
Let $v$ be a non-negative integer.
For a non-empty subset $I \subset \{1, \dotsc, m \}$ of cardinality $v+1$,
we define $D_I:= \cap_{i \in I} D_{i}$ (scheme-theoretic intersection).
If $D_I$ is non-empty, then it is purely of codimension $v$ in $Y$.
Moreover, we put
\[
Y^{(v)}:= \coprod_{I \subset \{1, \dotsc, m \}, \, \mathop{\mathrm{Card}}I=v+1}D_I.
\]

\begin{thm}[{Rapoport-Zink \cite[Satz 2.10]{Rapoport-Zink}, Saito \cite[Corollary 2.8]{Saito}}]\label{Theorem:weight spectral sequence}
Let the notation be as above.
Let $\ell \neq p$ be a prime number.
Let $\Lambda_\ell$ be $\Z/\ell^n\Z$, $\Z_{\ell}$, or $\Q_{\ell}$.
\begin{enumerate}
\renewcommand{\labelenumi}{(\roman{enumi})}
    \item We have a spectral sequence
\[
E^{v, w}_{1, \Lambda_\ell}=\bigoplus_{i \geq \max(0, -v)}H^{w-2i}_\et(Y^{(v+2i)}_{\overline{k}}, \Lambda_\ell(-i))\Rightarrow H^{v+w}_\et(X_{\overline{K}}, \Lambda_\ell),
\]
which is compatible with the action of $G_K$.
Here $(-i)$ denotes the Tate twist.
    \item Let $\sigma \in I_K$ be an element such that $t_{\ell}(\sigma) \in \Z_\ell(1)$ is a generator.
    There exists the following homomorphism of spectral sequences:
    \[\xymatrix{
&E^{v, w}_{1, \Lambda_\ell}=\bigoplus_{i \geq \max(0, -v)}H^{w-2i}_\et(Y^{(v+2i)}_{\overline{k}}, \Lambda_\ell(-i))\ar[d]_{1\otimes t_{\ell}(\sigma)}\ar@{=>}[r] &H^{v+w}_\et(X_{\overline{K}}, \Lambda_\ell)\ar[d]^{\sigma-1}\\
&E^{v+2, w-2}_{1, \Lambda_\ell}=\bigoplus_{i-1 \geq \max(0, -v-2)}H^{w-2i}_\et(Y^{(v+2i)}_{\overline{k}}, \Lambda_\ell(-i+1))\ar@{=>}[r] &H^{v+w}_\et(X_{\overline{K}}, \Lambda_\ell).
}
\] 
\end{enumerate}
\end{thm}

\begin{proof} 
For (i), see \cite[Satz 2.10]{Rapoport-Zink} and \cite[Corollary 2.8 (1)]{Saito}.
We remark that the spectral sequence constructed in \cite{Rapoport-Zink} coincides with that constructed in \cite{Saito} up to signs; see \cite[p.613]{Saito}.
In this paper, we use the spectral sequence constructed in \cite{Saito}.
The assertion (ii) follows from \cite[Corollary 2.8 (2)]{Saito}.
\end{proof}

The spectral sequence in Theorem \ref{Theorem:weight spectral sequence} is called the \textit{weight spectral sequence with $\Lambda_\ell$-coefficients}.

We will discuss the degeneracy of the weight spectral sequence.
For the weight spectral sequence with $\Q_\ell$-coefficients, we have the following theorem:

\begin{thm}\label{Theorem:degenerate}
The weight spectral sequence with $\Q_\ell$-coefficients degenerates at $E_2$ for every $\ell \neq p$.
\end{thm}

\begin{proof}
See \cite[Theorem 0.1]{Nakayama} or \cite[Theorem 1.1 (1)]{Ito-equal char}.
\end{proof}

For the weight spectral sequence with $\Lambda_\ell$-coefficients, where $\Lambda_{\ell}$ is either $\F_\ell$ or $\Z_{\ell}$,
we can prove the following theorem, which relies on the Weil conjecture and Theorem \ref{Theorem:Gabber}.

\begin{thm}\label{Theorem:torsion-freeness of weight spectral sequence}\
\begin{enumerate}
\renewcommand{\labelenumi}{(\roman{enumi})}
    \item Let $\Lambda_{\ell}$ be either $\F_\ell$ or $\Z_{\ell}$. For all but finitely many $\ell \neq p$, the weight spectral sequence with $\Lambda_{\ell}$-coefficients degenerates at $E_2$. \item For all but finitely many $\ell \neq p$, the $\Z_\ell$-modules $E^{v, w}_{1, \Z_\ell}$ and $E^{v, w}_{2, \Z_\ell}$ are torsion-free and
    we have
    $E^{v, w}_{2, \Z_\ell}\otimes_{\Z_\ell}\F_\ell \cong E^{v, w}_{2, \F_\ell}$ for all $v, w$.
\end{enumerate}
\end{thm}

\begin{proof}
The torsion-freeness of $E^{v, w}_{1, \Z_\ell}$ for all but finitely many $\ell \neq p$ follows from Theorem \ref{Theorem:Gabber}.

If $p=0$, by the comparison of \'etale and singular cohomology for varieties over $\C$,
it follows that the cokernel of the map
$d^{v, w}_1 \colon E^{v, w}_{1, \Z_{\ell}} \to E^{v+1, w}_{1, \Z_{\ell}}$
is torsion-free for all but finitely many $\ell \neq p$ and all $v, w$.
Thus Theorem \ref{Theorem:torsion-freeness of weight spectral sequence} is a consequence of Theorem \ref{Theorem:degenerate}.

We assume that $p > 0$.
First, we assume that $k$ is finitely generated over $\F_p$.
The family
$
\{ E^{v, w}_{1, \Lambda_{\ell}} \}_{\ell \neq p}
$
of $G_k$-modules over $\Lambda_{\ell}$
is of weight $w$ by Corollary \ref{Corollary:system of etale cohomology}.
Since $E^{v, w}_{2, \Lambda_{\ell}}$ is a subquotient of $E^{v, w}_{1, \Lambda_{\ell}}$,
the family
$
\{ E^{v, w}_{2, \Lambda_{\ell}} \}_{\ell \neq p}
$
is also of weight $w$.
Since the map $d^{v, w}_2 \colon E^{v, w}_{2, \Lambda_{\ell}} \to E^{v+2, w-1}_{2, \Lambda_{\ell}}$ is $G_k$-equivariant, it is a zero map for all but finitely many $\ell \neq p$ by Lemma \ref{Lemma:vanishing morphism}.
This fact proves the first assertion.

We shall prove the second assertion.
By the degeneracy of the weight spectral sequence with $\F_\ell$-coefficients, we have
\[
\sum_{v, w} \dim_{\F_\ell}{E^{v, w}_{2, \F_{\ell}}}=
\sum_{i} \dim_{\F_\ell}{H^{i}_\et(X_{\overline{K}}, \F_\ell)}
\]
for all but finitely many $\ell \neq p$.
By Theorem \ref{Theorem:degenerate},
we have
\[
\sum_{v, w} \rank_{\Z_\ell}{E^{v, w}_{2, \Z_{\ell}}}=
\sum_{i} \rank_{\Z_\ell}{H^{i}_\et(X_{\overline{K}}, \Z_\ell)}.
\]
By Theorem \ref{Theorem:Gabber} and Lemma \ref{Lemma:torsion-freeness of cohomology groups},
for all but finitely many $\ell \neq p$, we have
\[
\sum_{i} \rank_{\Z_\ell}{H^{i}_\et(X_{\overline{K}}, \Z_\ell)}=\sum_{i} \dim_{\F_\ell}{H^{i}_\et(X_{\overline{K}}, \F_\ell)}
\]
and
\[
\rank_{\Z_\ell}{E^{v, w}_{2, \Z_{\ell}}} \leq \dim_{\F_\ell}{E^{v, w}_{2, \F_{\ell}}}
\]
for all $v, w$.
It follows that, for all but finitely many $\ell \neq p$, we have
\[
\rank_{\Z_\ell}{E^{v, w}_{2, \Z_{\ell}}} = \dim_{\F_\ell}{E^{v, w}_{2, \F_{\ell}}}
\]
for all $v, w$.
Now the second assertion follows from Lemma \ref{Lemma:torsion-freeness of cohomology groups}.

The general case can be deduced from the case where $k$ is finitely generated over $\F_p$ by using N\'eron's blowing up as in \cite[Section 4]{Ito-equal char} and by using an argument in the proof of \cite[Lemma 3.2]{Ito-equal char}.
\end{proof}

In the rest of this section, we discuss
a relation between Conjecture \ref{Conjecture:A torsion analogue of the weight-monodromy conjecture} and the weight spectral sequence.

Let $\sigma \in I_K$ be an element such that,
for every $\ell \neq p$,
the image $t_{\ell}(\sigma) \in \Z_\ell(1)$ is a generator.
Let $i \geq 0$ be an integer.
The $i$-th iterate of
$(1 \otimes t_{\ell}(\sigma))^i$
induces a homomorphism
\[
(1 \otimes t_{\ell}(\sigma))^i \colon E^{-i, w+i}_{2, \Lambda_\ell} \to E^{i, w-i}_{2, \Lambda_\ell},
\]
see Theorem \ref{Theorem:weight spectral sequence} (ii).
Then we have the following conjecture.

\begin{conj}\label{Conjecture:weight-monodromy conjecture spectral sequence}
Let $\mathfrak{X}$ be a proper strictly semi-stable scheme  over $\Spec \O_K$ purely of relative dimension $d$.
Let the notation be as above.
We put $\Lambda_\ell=\Q_\ell$ (resp.\ $\Lambda_\ell=\F_\ell, \Z_\ell$).
Let $w$ be an integer.
Then for every $\ell \neq p$ (resp.\ all but finitely many $\ell \neq p$),
the above morphism
$
(1 \otimes t_{\ell}(\sigma))^i \colon E^{-i, w+i}_{2, \Lambda_\ell} \to E^{i, w-i}_{2, \Lambda_\ell}
$
is an isomorphism for every $i \geq 0$.
\end{conj}

\begin{rem}
Assume that $p=0$.
Then Conjecture \ref{Conjecture:weight-monodromy conjecture spectral sequence} for
$(\mathfrak{X}, \Lambda_\ell=\Q_\ell)$
is true; see \cite[Theorem 1.1 (2)]{Ito-equal char}.
Therefore,
by a similar argument as in the proof of Theorem \ref{Theorem:torsion-freeness of weight spectral sequence},
we see that Conjecture \ref{Conjecture:weight-monodromy conjecture spectral sequence} for $(\mathfrak{X}, \Lambda_\ell=\F_\ell)$ and $(\mathfrak{X}, \Lambda_\ell=\Z_\ell)$ also holds.
\end{rem}

\begin{lem}\label{Lem:implication}
Conjecture \ref{Conjecture:weight-monodromy conjecture spectral sequence} for
$(\mathfrak{X}, w, \Lambda_\ell=\F_\ell)$ is equivalent to Conjecture \ref{Conjecture:weight-monodromy conjecture spectral sequence} for
$(\mathfrak{X}, w, \Lambda_\ell=\Z_\ell)$.
\end{lem}

\begin{proof}
By Theorem \ref{Theorem:torsion-freeness of weight spectral sequence}, it follows that,
for all but finitely many $\ell \neq p$,
the $\Z_\ell$-module $E^{v, w}_{2, \Z_\ell}$ is torsion-free and
we have
$E^{v, w}_{2, \Z_\ell}\otimes_{\Z_\ell}\F_\ell \cong E^{v, w}_{2, \F_\ell}$
for all $v, w$.
Therefore the assertion follows from Nakayama's lemma.
\end{proof}

In the rest of this section, we assume that $K$ is a non-archimedean local field.
Let $q$ be the number of elements in the residue field of $K$.

\begin{rem}\label{Remark:weight-monodromy equivalent spectral sequence}
It is well known that
Conjecture \ref{Conjecture:Weight-monodromy} for $(X, w)$ is equivalent to
Conjecture \ref{Conjecture:weight-monodromy conjecture spectral sequence} for
$(\mathfrak{X}, w, \Lambda_\ell=\Q_\ell)$; see \cite[Proposition 2.5]{Ito-equal char} for example.
\end{rem}

Similarly to Remark \ref{Remark:weight-monodromy equivalent spectral sequence},
we have the following lemma.

\begin{lem}\label{Lemma:equivalent}
Let $\mathfrak{X}$ be a proper strictly semi-stable scheme  over $\Spec \O_K$ purely of relative dimension $d$
with generic fiber $X$
and let $w$ be an integer.
Then Conjecture \ref{Conjecture:A torsion analogue of the weight-monodromy conjecture} for
$(X, w)$ is equivalent to
Conjecture \ref{Conjecture:weight-monodromy conjecture spectral sequence} for
$(\mathfrak{X}, w, \Lambda_\ell=\F_\ell)$.
\end{lem}

\begin{proof}
By Theorem \ref{Theorem:torsion-freeness of weight spectral sequence}, the weight spectral sequence with $\F_\ell$-coefficients degenerates at $E_2$ for all but finitely many $\ell \neq p$.
Hence the claim follows from
Lemma \ref{Lemma:relatively prime}, Corollary \ref{Corollary:system of etale cohomology},
Theorem \ref{Theorem:weight spectral sequence} (ii),
and the definition of the monodromy filtration.
\end{proof}

Finally, we record the following well known proposition.

\begin{prop}\label{Proposition:polynomial Frob vanishing}
Let $Z$ be a proper smooth scheme over $\Spec K$ and $w$ an integer.
Let $\Frob \in G_K$ be a lift of the geometric Frobenius element.
\begin{enumerate}
\renewcommand{\labelenumi}{(\roman{enumi})}
    \item There is a non-zero monic polynomial $P(T) \in \Z[T]$
    such that,
    for all but finitely many $\ell \neq p$,
    we have
    $P(\Frob)=0$ on $H^{w}_\et(Z_{\overline{K}}, \Z_\ell)$.
    \item For every $\ell \neq p$,
    there exists a unique increasing, separated, exhaustive filtration
    \[
    \{ W_{i, \Q_\ell} \}_i
    \]
    on $H^{w}_\et(Z_{\overline{K}}, \Q_\ell)$ which is stable by the action of $G_K$
    and satisfies the following property.
    For every $i$,
    there exists a Weil $q^{w+i}$-polynomial $P_i(T) \in \Z[T]$ such that $P_i(\Frob)=0$ on the $i$-th graded piece
    $\Gr^{W}_{i, \Q_\ell}:=W_{i, \Q_\ell}/W_{i-1, \Q_\ell}$.
    Moreover, we can take the polynomial $P_i(T) \in \Z[T]$
    independent of $\ell \neq p$.
    \item Assume that Conjecture \ref{Conjecture:Weight-monodromy} for $(Z, w)$ is true. Then, for every $i$, there exists a Weil $q^{w+i}$-polynomial $P_i(T) \in \Z[T]$
    such that,
    for every $\ell \neq p$,
    we have
    $P_i(\Frob)=0$ on the $i$-th graded piece $\Gr^{M}_{i, \Q_\ell}$ of the monodromy filtration on $H^{w}_\et(Z_{\overline{K}}, \Q_\ell)$.
\end{enumerate}
\end{prop}

\begin{proof}
We may assume that $Z$ is geometrically connected.
By de Jong's alteration \cite[Theorem 6.5]{deJong:Alteration},
there exist a finite extension $L$ of $K$
and a proper strictly semi-stable scheme $\mathfrak{X}$ over $\Spec \O_L$ such that
the generic fiber $X$ of $\mathfrak{X}$ is an alteration of $Z$.
We may assume further that $K=L$ and $X$ is geometrically connected.
By a trace argument, we see that
$H^{w}_\et(Z_{\overline{K}}, \Q_\ell)$ is a direct summand of
$H^{w}_\et(X_{\overline{K}}, \Q_\ell)$
as a $G_K$-representation for every $\ell \neq p$.

Let $\{ F^i_{\Q_\ell} \}_i$ be the decreasing filtration on
$H^{w}_\et(X_{\overline{K}}, \Q_\ell)$
arising from the weight spectral sequence.
The filtration $\{ F^i_{\Q_\ell} \}_i$ defines a decreasing filtration on
$H^{w}_\et(Z_{\overline{K}}, \Q_\ell)$, which is also denoted by $\{ F^i_{\Q_\ell} \}_i$.
Let $\{ W_{i, \Q_\ell} \}_i$ be the increasing filtration on $H^{w}_\et(Z_{\overline{K}}, \Q_\ell)$ defined by
\[
W_{i, \Q_\ell}:=F^{-i}_{\Q_\ell}.
\]
Since
the $i$-th graded piece
$\Gr^{W}_{i, \Q_\ell}:=W_{i, \Q_\ell}/W_{i-1, \Q_\ell}$ is a subquotient of $E^{-i, w+i}_{1, \Q_\ell}$,
by the Weil conjecture,
there exists a Weil $q^{w+i}$-polynomial $P_i(T) \in \Z[T]$
such that,
for every $\ell \neq p$,
we have
$P_i(\Frob)=0$ on $\Gr^{W}_{i, \Q_\ell}$.
Thus the assertion (ii) follows.

The assertion (i) follows from (ii) and Theorem \ref{Theorem:Gabber}.
If Conjecture \ref{Conjecture:Weight-monodromy} for $(Z, w)$ is true, the filtration
$\{ W_{i, \Q_\ell} \}_i$
coincides with the monodromy filtration
$\{ M_{i, \Q_\ell} \}_i$.
Therefore the assertion (iii) follows from (ii).
\end{proof}

\begin{rem}\label{Remark:weight filtration}
We call the filtration $\{ W_{i, \Q_\ell} \}_i$ in Proposition \ref{Proposition:polynomial Frob vanishing} the
\textit{weight filtration} on $H^{w}_\et(Z_{\overline{K}}, \Q_\ell)$.
(The numbering used here differs from it of \cite[Proposition-d\'efinition (1.7.5)]{WeilII}.)
\end{rem}

\begin{rem}\label{Remark:l-independence of characteristic poly}
Let $\Frob \in G_K$ be a lift of the geometric Frobenius element.
Let $Z$ be a proper smooth scheme over $K$.
It is conjectured that
the characteristic polynomial $P_{\Frob, \ell}(T)$ of $\Frob$ acting on
$H^{w}_\et(Z_{\overline{K}}, \Q_\ell)$
is in $\Z[T]$ and independent of $\ell \neq p$.
If $Z$ is a surface or $K$ is of equal characteristic,
this conjecture is true; see \cite[Corollary 2.5]{Ochiai} and \cite[Theorem 1.4]{LZ}.
(See also \cite[Theorem 3.3]{Terasoma}.)
If this conjecture and Conjecture \ref{Conjecture:Weight-monodromy} for $(Z, w)$ are true,
then we can take $P_i(T)$ in Proposition \ref{Proposition:polynomial Frob vanishing} (iii) as the characteristic polynomial of $\Frob$ acting on
$\Gr^{M}_{i, \Q_\ell}$.
\end{rem}

\section{Equal characteristic cases}\label{Section:Proof of the main theorem: equal and complete intersection}

In this section, we will prove Theorem \ref{Theorem:A torsion analogue of the weight-monodromy conjecture} in the case (\ref{equal}).
We will use the language of ultraproducts following \cite{Cadoret}.
We first recall some properties of ultraproducts which we need.
For details, see \cite[Appendix]{CHT} for example.
The notation used here is similar to that of \cite{Cadoret}.

\subsection{Ultraproducts}\label{Subsection:Ultraproducts}
Let $\mathcal{L}$ be an infinite set of prime numbers.
We define
\[
\underline{F} := \prod_{\ell \in \mathcal{L}} \overline{\F}_\ell,
\]
where $\overline{\F}_\ell$ is an algebraic closure of $\F_\ell$.
For a subset $S \subset \mathcal{L}$, let $e_S$ be the characteristic function of
$\mathcal{L} \backslash S$, which we consider as an element of $\underline{F}$.
Attaching to an ultrafilter $\mathfrak{u}$ on $\mathcal{L}$
a prime ideal
\[
\mathfrak{m}_{\mathfrak{u}}:=\langle e_S \ \vert \ S \in \mathfrak{u}  \rangle \subset \underline{F}
\]
defines a bijection from the set of ultrafilters on $\mathcal{L}$ to $\Spec \underline{F}$.
Note that every prime ideal of $\underline{F}$ is a maximal ideal.
We say that an ultrafilter $\mathfrak{u}$ on $\mathcal{L}$ is principal if it
corresponds to a principal ideal.
For a non-principal ultrafilter $\mathfrak{u}$,
we define
\[
\overline{\Q}_\mathfrak{u}:= \underline{F}/\mathfrak{m}_{\mathfrak{u}}.
\]
It is a field of characteristic $0$ and is isomorphic to the field of complex numbers $\C$.
The field $\overline{\Q}_\mathfrak{u}$ is called the \textit{ultraproduct} of
$\{ \overline{\F}_\ell \}_{\ell \in \mathcal{L}}$
(with respect to the non-principal ultrafilter $\mathfrak{u}$).
The ring homomorphism $\underline{F} \to \overline{\Q}_\mathfrak{u}$ is flat; see \cite[Lemma in Section 4.1.4]{CHT}.

\begin{rem}\label{Remark:ultraproduct}
Let $\mathcal{L}' \subset \mathcal{L}$ be a subset such that
$\mathcal{L} \backslash \mathcal{L}'$ is finite.
The projection
$\underline{F} \to \prod_{\ell \in \mathcal{L}'} \overline{\F}_\ell$
defines a bijection from the set of non-principal ultrafilters on $\mathcal{L}'$ to the set of non-principal ultrafilters on $\mathcal{L}$.
\end{rem}

Let
$\{ M_\ell \}_{\ell \in \mathcal{L}}$
be a family of $\overline{\F}_\ell$-vector spaces.
We define
\[
\underline{M}:=\prod_{\ell \in \mathcal{L}} M_\ell.
\]
For the $\underline{F}$-module $\underline{M}$, the following assertions are equivalent.
\begin{itemize}
    \item $\underline{M}$ is a finitely generated $\underline{F}$-module.
    \item $\underline{M}$ is a finitely presented $\underline{F}$-module.
    \item $\sup_{\ell \in \mathcal{L}} \dim_{\overline{\F}_\ell} M_\ell < \infty$.
\end{itemize}
We put $M_{\mathfrak{u}}:=\underline{M} \otimes_{\underline{F}} \overline{\Q}_\mathfrak{u}$ for a non-principal ultrafilter $\mathfrak{u}$.
We will use a similar notation for a family
$\{ f_\ell \}_{\ell \in \mathcal{L}}$
of maps of $\overline{\F}_\ell$-vector spaces.

\begin{lem}\label{Lemma:module ultraproduct}
Let $\{ M_\ell \}_{\ell \in \mathcal{L}}$
and $\{ N_\ell \}_{\ell \in \mathcal{L}}$ be families of $\overline{\F}_\ell$-vector spaces.
Assume that $\underline{M}$ and $\underline{N}$ are finitely generated $\underline{F}$-modules.
Let $\{ f_\ell \}_{\ell \in \mathcal{L}}$ be a family of maps $f_\ell \colon M_\ell \to N_\ell$
of $\overline{\F}_\ell$-vector spaces.
Then the following assertions are equivalent.
\begin{enumerate}
\renewcommand{\labelenumi}{(\roman{enumi})}
    \item $f_{\mathfrak{u}} \colon M_{\mathfrak{u}} \to N_{\mathfrak{u}}$ is an isomorphism for every non-principal ultrafilter $\mathfrak{u}$.
    \item $f_\ell \colon M_\ell \to N_\ell$ is an isomorphism for all but finitely many $\ell \in \mathcal{L}$.
\end{enumerate}
\end{lem}

\begin{proof}
For a subset $S \subset \mathcal{L}$ which is contained in every non-principal ultrafilter,
the complement $\mathcal{L} \backslash S$ is finite.
Hence the lemma follows from \cite[Lemma 4.3.3]{CHT}.
\end{proof}

Let $p$ be a prime number and
let $\mathcal{L}$ be the set of prime numbers $\ell \neq p$.
Let $K$ be a Henselian discrete valuation field.
Assume that the characteristic of the residue field $k$ of $K$ is $p$.
Let $\mathfrak{X}$ be a proper strictly semi-stable scheme  over $\O_K$ purely of relative dimension $d$.
We retain the notation of Section \ref{Section:Torsion-freeness of the weight spectral sequence}.

Let $\mathfrak{u}$ be a non-principal ultrafilter on $\mathcal{L}$.
Since the map
$\underline{F} \to \overline{\Q}_\mathfrak{u}$
is flat, we have the following weight spectral sequence with $\overline{\Q}_\mathfrak{u}$-coefficients:
\[
E^{v, w}_{1, \overline{\Q}_\mathfrak{u}}= \bigoplus_{i \geq \max(0, -v)}  H^{w-2i}_\et(Y^{(v+2i)}_{\overline{k}}, \overline{\Q}_\mathfrak{u}(-i))\Rightarrow H^{v+w}_\et(X_{\overline{K}}, \overline{\Q}_\mathfrak{u}).
\]
Here we define
\[
H^{w}_\et(X_{\overline{K}}, \overline{\Q}_\mathfrak{u}):= (\prod_{\ell \neq p} H^{w}_\et(X_{\overline{K}}, \overline{\F}_\ell)) \otimes_{\underline{F}} \overline{\Q}_\mathfrak{u},
\]
and similarly for $H^{w-2i}_\et(Y^{(v+2i)}_{\overline{k}}, \overline{\Q}_\mathfrak{u}(-i))$.
For an element $\sigma \in I_K$
such that,
for every $\ell \neq p$,
the image $t_{\ell}(\sigma) \in \Z_\ell(1)$ is a generator,
we have a monodromy operator
\[
(1 \otimes t_{\ell}(\sigma))^i \colon E^{-i, w+i}_{2, \overline{\Q}_{\mathfrak{u}}} \to E^{i, w-i}_{2, \overline{\Q}_{\mathfrak{u}}}
\]
for all $w, i \geq 0$.

\begin{lem}\label{Lemma:equivalent ultraproduct}
Conjecture \ref{Conjecture:weight-monodromy conjecture spectral sequence} for
$(\mathfrak{X}, w, \Lambda_\ell =\F_\ell)$ is equivalent to the assertion that the morphism
\[
(1 \otimes t_{\ell}(\sigma))^i \colon E^{-i, w+i}_{2, \overline{\Q}_{\mathfrak{u}}} \to E^{i, w-i}_{2, \overline{\Q}_{\mathfrak{u}}}
\]
is an isomorphism for every non-principal ultrafilter $\mathfrak{u}$ on $\mathcal{L}$ and every $i \geq 0$.
\end{lem}

\begin{proof}
The $\underline{F}$-module
$\prod_{\ell \neq p} (E^{v, w}_{2, \F_\ell} \otimes_{\F_\ell} \overline{\F}_\ell)$
is finitely generated
for all $v, w$
by Theorem \ref{Theorem:Gabber}.
Hence the assertion follows from Lemma \ref{Lemma:module ultraproduct}.
\end{proof}

Finally, we define an ultraproduct variant of the notion of weight.
Let $k$ be a finitely generated field over $\F_p$ and let $\mathfrak{u}$ be a non-principal ultrafilter on $\mathcal{L}$.
Let $\{ H_{\ell} \}_{\ell \in \mathcal{L}}$ be a family of finite dimensional $G_k$-representations over $\overline{\F}_{\ell}$
such that
the $\underline{F}$-module
$\underline{H}$ is finitely generated.
Then
$H_{\mathfrak{u}}$ is a
finite dimensional representation of $G_k$ over $\overline{\Q}_{\mathfrak{u}}$.
(We do not impose any continuity conditions here.)
Let $w$ be an integer.
Let
$\iota \colon \overline{\Q}_{\mathfrak{u}} \cong \C$
be an isomorphism.
We say that $H_{\mathfrak{u}}$ is \textit{$\iota$-pure of weight} $w$ if the following conditions are satisfied:
\begin{itemize}
    \item There is an integral scheme $U$ of finite type over $\F_p$ with function field $k$
such that the family $\{ H_{\ell} \}_{\ell \in \mathcal{L}}$
comes from a family
$\{ \mathcal{F}_\ell \}_{\ell \in \mathcal{L}}$
of locally constant constructible $\overline{\F}_{\ell}$-sheaves
on $U$.
    \item Moreover, for every closed point $x \in U$ and for
every eigenvalue $\alpha$ of $\Frob_x$ acting on $H_{\mathfrak{u}}$, we have
$\vert \iota(\alpha) \vert = (q_x)^{w/2}$.
\end{itemize}
(See also Section \ref{Subsection:Weights}.)
We say that $H_{\mathfrak{u}}$ is \textit{pure of weight} $w$ if it is $\iota$-pure of weight $w$ for every $\iota \colon \overline{\Q}_{\mathfrak{u}} \cong \C$.

\subsection{Proof of Theorem \ref{Theorem:A torsion analogue of the weight-monodromy conjecture} in the case (\ref{equal})}\label{Subsection:The case equal}

We shall prove Theorem \ref{Theorem:A torsion analogue of the weight-monodromy conjecture} in the case (\ref{equal}).
By de Jong's alteration \cite[Theorem 6.5]{deJong:Alteration},
a trace argument, and Lemma \ref{Lemma:equivalent},
it suffices to prove the following theorem.

\begin{thm}\label{Theorem:equal char weight-monodromy conjecture spectral sequence}
Let $K$ be a Henselian discrete valuation field of equal characteristic $p>0$.
Then Conjecture \ref{Conjecture:weight-monodromy conjecture spectral sequence} for $\Lambda_\ell=\F_\ell$ is true.
\end{thm}

\begin{proof}
We use the same strategy as in \cite{Ito-equal char}.
The only problem is that we cannot use Weil II \cite{WeilII} directly since it works with \'etale cohomology with $\Q_\ell$-coefficients.
However, Cadoret established an ultraproduct variant of Weil II in \cite{Cadoret}.
By using her results, the same arguments as in \cite{Ito-equal char} can be carried out.
We recall the arguments for the reader's convenience.

Let $\mathcal{L}$ be the set of prime numbers $\ell \neq p$.
Let $\mathfrak{X}$ be a proper strictly semi-stable scheme over $\O_K$ purely of relative dimension $d$
with generic fiber $X$.
We retain the notation of Section \ref{Subsection:Ultraproducts}.
By Lemma \ref{Lemma:equivalent ultraproduct},
it suffices to prove that the morphism
$
(1 \otimes t_{\ell}(\sigma))^i \colon E^{-i, w+i}_{2, \overline{\Q}_{\mathfrak{u}}} \to E^{i, w-i}_{2, \overline{\Q}_{\mathfrak{u}}}
$
is an isomorphism for every non-principal ultrafilter $\mathfrak{u}$ on $\mathcal{L}$ and for all $w, i \geq 0$.

By using N\'eron's blowing up as in \cite[Section 4]{Ito-equal char} and by using an argument in the proof of \cite[Lemma 3.2]{Ito-equal char},
we may assume that
there exist a connected smooth scheme $\Spec A$ over $\F_p$ and an element $\varpi \in A$ satisfying the following properties:
\begin{itemize}
    \item $D:= \Spec A/(\varpi)$ is an irreducible divisor on $A$ which is smooth over $\F_p$ and $\O_K$ is the Henselization of the local ring of $\Spec A$ at the prime ideal $(\varpi) \subset A$.
    \item There is a proper scheme $\widetilde{\mathfrak{X}}$ over $\Spec A$ which is smooth over $\Spec A[1/\varpi]$ such that
$\widetilde{\mathfrak{X}} \otimes_A \O_K \cong \mathfrak{X}$.
\end{itemize}
Let $f \colon \widetilde{\mathfrak{X}} \to \Spec A$ be the structure morphism.
The function field of $D$ is the residue field $k$ of $K$, which is finitely generated over $\F_p$.

Let $w \geq 0$ be an integer.
By the same construction as in Section \ref{Subsection:A torsion analogue of the weight-monodromy conjecture},
after removing finitely many $\ell \neq p$ from $\mathcal{L}$,
we can construct
the monodromy filtration
$\{ M_{i, \F_\ell} \}_i$
with coefficients in $\F_\ell$
on $H^{w}_\et(X_{\overline{K}}, \F_\ell)$
for every $\ell \in \mathcal{L}$.
We have
$\sup_{\ell \in \mathcal{L}} \dim_{\F_\ell} \Gr^{M}_{i, \F_\ell} < \infty$, where
$\Gr^{M}_{i, \F_\ell}:=M_{i, \F_\ell}/M_{i-1, \F_\ell}$
is the $i$-th graded piece.
Let $\mathfrak{u}$ be a non-principal ultrafilter on $\mathcal{L}$.
By an analogue of Lemma \ref{Lemma:equivalent},
it suffices to prove that
the $G_k$-representation over $\overline{\Q}_{\mathfrak{u}}$
\[
(\prod_{\ell \in \mathcal{L}} \Gr^{M}_{i, \F_\ell} \otimes_{\F_\ell} \overline{\F}_\ell) \otimes_{\underline{F}} \overline{\Q}_{\mathfrak{u}}
\]
is pure of weight $w+i$ for every $i$.

By applying a construction given in \cite[Variante (1.7.8)]{WeilII} to
the higher direct image sheaf
$
R^wf_*\F_\ell
$
and by using a similar construction as in Section \ref{Subsection:A torsion analogue of the weight-monodromy conjecture},
after removing finitely many $\ell \neq p$ from $\mathcal{L}$,
we get a locally constant constructible $\overline{\F}_\ell$-sheaf
$\mathcal{F}_\ell[D]$
on $D$ with a filtration
$\{ \mathcal{M}_{i, \ell} \}_i$.
For every $i$,
the stalk of the quotient
\[
\Gr^{\mathcal{M}}_{i, \ell}:=\mathcal{M}_{i, \ell}/\mathcal{M}_{i-1, \ell}
\]
at the geometric generic point of $U$ is isomorphic to
$\Gr^{M}_{i, \F_\ell} \otimes_{\F_\ell} \overline{\F}_\ell$
as a $G_k$-representation for every $\ell \in \mathcal{L}$.

Let $x \in D$
be a closed point.
We can find a connected smooth curve
$C \subset \Spec A$ over $\F_p$ such that
$C \cap D = \{ x \}$
and the image of $\varpi \in A$ in $\O_{C, x}$ is a uniformizer.
Let $L$ be the field of fractions of the completion $\widehat{\O}_{C, x}$ of $\O_{C, x}$.
We write $Z:= \widetilde{\mathfrak{X}} \otimes_A L$.
By the construction,
for all but finitely many $\ell \in \mathcal{L}$,
the stalk
$(\Gr^{\mathcal{M}}_{i, \ell})_{\overline{x}}$ is isomorphic to the base change of the $i$-th graded piece of the monodromy filtration with
coefficients in $\F_\ell$ on $H^{w}_\et(Z_{\overline{L}}, \F_\ell)$ as a $G_{\kappa(x)}$-representation.
Thus
we see that
$
(\prod_{\ell \in \mathcal{L}}(\Gr^{\mathcal{M}}_{i, \ell})_{\overline{x}}) \otimes_{\underline{F}} \overline{\Q}_{\mathfrak{u}}
$
is pure of weight $w+i$ by \cite[Corollary 5.3.2.4]{Cadoret} together with Corollary \ref{Corollary:system of etale cohomology} and \cite[Lemma in 11.3]{Cadoret}.
This fact completes the proof of Theorem \ref{Theorem:equal char weight-monodromy conjecture spectral sequence}.
\end{proof}

\section{The case of set-theoretic complete intersections in toric varieties}\label{Section:Set-theoretic complete intersections cases in toric varieties}

In this section, we will prove Theorem \ref{Theorem:A torsion analogue of the weight-monodromy conjecture} in the case (\ref{complete intersection}).
Let $K$ be a non-archimedean local field with ring of integers $\O_K$.
Let $k$ be the residue field of $\O_K$.
Let $p > 0$ be the characteristic of $k$.
Since we have already shown that
Theorem \ref{Theorem:A torsion analogue of the weight-monodromy conjecture} holds in the case (\ref{equal}),
we may assume that $K$ is of characteristic $0$.

\subsection{A uniform variant of a theorem of Huber}\label{Subsection:A uniform variant of a theorem of Huber}

We will recall a result from \cite{Ito20} which will be used in the proof of Theorem \ref{Theorem:A torsion analogue of the weight-monodromy conjecture} in the case (\ref{complete intersection}).
In this section,
we will freely use the theory of adic spaces developed by Huber.
The theory of \'etale cohomology for adic spaces was developed in \cite{Huber96}.

Let $\C_p$ be the completion of an algebraic closure $\overline{K}$ of $K$,
which is a complete non-archimedean field (in the sense of \cite[Definition 1.1.3]{Huber96}).
Let $L \subset \C_p$ be a subfield
such that it is also a complete non-archimedean field
with the induced topology.
Let $\O_L$ be the ring of integers of $L$.
For a scheme $X$ of finite type over $L$,
the adic space associated with $X$ is denoted by
\[
X^{\ad}:= X \times_{\Spec L} \Spa(L, \O_L);
\]
see \cite[Proposition 3.8]{Huber94}.
For an adic space $Y$ locally of finite type over $\Spa(L, \O_L)$,
we denote by 
\[
Y_{\C_p}:=Y \times_{\Spa(L, \O_L)} \Spa(\C_p, \O_{\C_p})
\]
the base change of $Y$ to $\Spa(\C_p, \O_{\C_p})$.

Let us recall the following theorem due to Huber:
\begin{itemize}
    \item Let $Y$ be a proper scheme over $L$ and $X \hookrightarrow Y$ a closed immersion.
We have a closed immersion $X^{\ad} \hookrightarrow Y^{\ad}$ of adic spaces over $\Spa(L, \O_L)$.
We fix a prime number $\ell \neq p$.
Then,
there is an open subset $V$ of $Y^{\ad}$
containing
$X^{\ad}$
such that
the pull-back map
\[
H^w_\et(V_{\C_p}, \F_\ell) \to H^w_\et(X^{\ad}_{\C_p}, \F_\ell)
\]
of \'etale cohomology groups is an isomorphism for every $w$.
\end{itemize}
(See \cite[Theorem 3.6]{Huber98b} for a more general result.)
Scholze used this theorem in his proof of the weight-monodromy conjecture in the case (\ref{complete intersection}).

In our case,
we need the following uniform variant of Huber's theorem:

\begin{thm}[{\cite[Corollary 4.11]{Ito20}}]\label{Theorem:uniform tubular nbd}
Let $Y$ be a proper scheme over $L$ and $X \hookrightarrow Y$ a closed immersion.
We have a closed immersion $X^{\ad} \hookrightarrow Y^{\ad}$ of adic spaces over $\Spa(L, \O_L)$.
Then,
there is an open subset $V$ of $Y^{\ad}$
containing
$X^{\ad}$
such that,
for every prime number $\ell \neq p$,
the pull-back map
\[
H^w_\et(V_{\C_p}, \F_\ell) \to H^w_\et(X^{\ad}_{\C_p}, \F_\ell)
\]
is an isomorphism for every $w$.
\end{thm}
\begin{proof}
See \cite[Corollary 4.11]{Ito20}.
\end{proof}

\subsection{Proof of Theorem \ref{Theorem:A torsion analogue of the weight-monodromy conjecture} in the case (\ref{complete intersection})}\label{Subsection:The case complete intersection}

The proof is the same as that of \cite[Theorem 9.6]{Scholze},
except that we use Theorem \ref{Theorem:uniform tubular nbd} instead of Huber's theorem.
We shall give a sketch here.
We will use the terminology in \cite{Scholze}.

Let $X$ be a geometrically connected projective smooth scheme over $K$
which is a set-theoretic complete intersection
in a projective smooth toric variety
$Y_{\Sigma, K}$ over $K$ associated with a fan $\Sigma$.
After replacing $K$ by its finite extension,
we may assume that
the action of $I_K$
on $H^{w}_\et(X_{\overline{K}}, \F_\ell)$ is unipotent
and factors through
$
t_{\ell}
$
for every $w$ and for every $\ell \neq p$.

Let $\varpi$ be a uniformizer of $K$.
We fix a system
$\{ \varpi^{1/p^n} \}_{n \geq 0} \subset \overline{K}$
of $p^{n}$-th roots of $\varpi$.
Let $L$ be the completion of
$
\bigcup_{n \geq 0}K(\varpi^{1/p^n}),
$
which is a perfectoid field.
Let $G_L=\Aut(\overline{L}/L)$ be the absolute Galois group of $L$,
where
$\overline{L}$ is the algebraic closure of $L$ in $\C_p$.
Then we have a surjection
$G_L \to G_k$.
Thus there exists
a lift $\Frob \in G_L$ of the geometric Frobenius element.
Let $I_L$ be the kernel of the map $G_L \to G_k$.
We have $I_L \subset I_K$.
Since $\bigcup_{n \geq 0}K(\varpi^{1/p^n})$ is a
pro-$p$ extension of $K$,
there exists an element
$\sigma \in I_L$ such that,
for every $\ell \neq p$,
the image $t_\ell(\sigma) \in \Z_\ell(1)$
is a generator.
In other words,
there exists an element $\sigma \in I_L$
such that it defines the monodromy filtration
with coefficients in $\F_\ell$ on $H^{w}_\et(X_{\overline{K}}, \F_\ell)$
for all but finitely many $\ell \neq p$.
Therefore,
it suffices to prove a natural analogue of
Theorem \ref{Theorem:A torsion analogue of the weight-monodromy conjecture}
for the family
$\{ H^{w}_\et(X_{\overline{K}}, \F_\ell) \}_{\ell \neq p}$
of $G_L$-representations.
Moreover, in order to prove this, we can replace $L$ by its finite extension if necessary.

Let $L^\flat$ be the tilt of $L$.
We have an identification
$G_L=G_{L^\flat}$.
The choice of the system $\{ \varpi^{1/p^n} \}_{n \geq 0} \subset \overline{K}$
gives an identification between $L^{\flat}$ and the completion of the perfection of
the field of formal Laurent series $k((t))$ over $k$.

Let
$Y_{\Sigma, L}$ be the toric variety over $L$ associated with the fan $\Sigma$ and
let $Y^{\ad}_{\Sigma, L}$ be the adic space associated with $Y_{\Sigma, L}$.
We define $Y^{\ad}_{\Sigma, L^\flat}$ similarly.
By \cite[Theorem 8.5 (iii)]{Scholze},
we have a projection
\[
\pi \colon Y^{\ad}_{\Sigma, L^\flat} \to Y^{\ad}_{\Sigma, L}
\]
of topological spaces.
By Theorem \ref{Theorem:uniform tubular nbd},
there exists an open subset $V$ of $Y^{\ad}_{\Sigma, L}$
containing $X^{\ad}_L$ such that,
for every prime number $\ell \neq p$,
the pull-back map
\[
H^w_\et(V_{\C_p}, \F_\ell) \to H^w_\et(X^{\ad}_{\C_p}, \F_\ell)
\]
is an isomorphism for every $w$.
By \cite[Corollary 8.8]{Scholze},
there exists
a closed subscheme $Z$ of $Y_{\Sigma, L^\flat}$,
which is defined over a global field (i.e.\ the function field of a smooth connected curve over $k$),
such that $Z^\ad$ is contained in $\pi^{-1}(V)$ and
$\dim Z = \dim X$.
We may assume that $Z$ is irreducible.
By \cite[Theorem 4.1]{deJong:Alteration},
there exists an alteration $Z' \to Z$,
which is defined over a global field,
such that $Z'$ is projective and smooth over $L^\flat$.
(We note that $L^\flat$ is a perfect field.)
We may assume further that $Z'$ and $Z$ are geometrically irreducible.

We have the following composition for every $\ell \neq p$ and every $w$:
\[
H^w_\et(X^\ad_{\C_p}, \F_\ell) \overset{\cong}{\to} H^w_\et(V_{\C_p}, \F_\ell) \to H^w_\et(\pi^{-1}(V)_{\C^{\flat}_p}, \F_\ell) \to H^w_\et(Z^\ad_{\C^{\flat}_p}, \F_\ell)
 \to H^w_\et((Z'_{\C^{\flat}_p})^\ad, \F_\ell),
\]
where the first map is the inverse map of the pull-back map, the second map is induced by \cite[Theorem 8.5 (v)]{Scholze}, and the last two maps are the pull-back maps.
By using a comparison theorem of Huber \cite[Theorem 3.8.1]{Huber96},
we obtain a map
\[
H^w_\et(X_{\C_p}, \F_\ell) \to H^w_\et(Z'_{\C^{\flat}_p}, \F_\ell)
\]
for every $\ell \neq p$ and every $w$.
This map is compatible with the actions of $G:=G_L=G_{L^\flat}$ on both sides and compatible with the cup products.

For $w=2 \dim X$,
by the same argument as in the proof of \cite[Lemma 9.9]{Scholze},
we conclude that
the above map is an isomorphism for all but finitely many $\ell \neq p$
from the fact that
the image of
the $(\dim X)$-th power of the Chern class of an ample line bundle on
$Y_{\Sigma, \C^\flat_p}$
under the map
\[
H^{2 \dim X}_\et(Y_{\Sigma, \C^\flat_p}, \F_\ell) \to H^{2 \dim X}_\et(Z'_{\C^{\flat}_p}, \F_\ell)
\]
is not zero for all but finitely many $\ell \neq p$.
By Poincar\'e duality,
it follows that
$H^w_\et(X_{\C_p}, \F_\ell)$
is a direct summand of
$H^w_\et(Z'_{\C^{\flat}_p}, \F_\ell)$
as a $G$-representation
for every $w$ and for all but finitely many $\ell \neq p$.
Since $Z'$ is defined over a global field,
a natural analogue of
Theorem \ref{Theorem:A torsion analogue of the weight-monodromy conjecture}
holds
for the family
$\{ H^{w}_\et(Z'_{\C^{\flat}_p}, \F_\ell) \}_{\ell \neq p}$
of $G$-representations
by the case (\ref{equal}).
This fact completes the proof of Theorem \ref{Theorem:A torsion analogue of the weight-monodromy conjecture} in the case (\ref{complete intersection}).

\section{The case of abelian varieties}\label{Section:Proof of the main theorem: abelian variety}

\subsection{Proof of Theorem \ref{Theorem:A torsion analogue of the weight-monodromy conjecture} in the case (\ref{abelian variety})}\label{Subsection:The case abelian variety}

We use the same notation as in Section \ref{Section:A torsion analogue of the weight-monodromy conjecture}.
Let $A$ be an abelian variety over $K$.
Let $\mathscr{A}$ be the N\'eron model of $A$.
After replacing $K$ by its finite extension, we may assume
that $A$ has semi-abelian reduction, i.e.\ the identity component $\mathscr{A}^0_s$ of the special fiber $\mathscr{A}_s$ of $\mathscr{A}$ is a semi-abelian variety over $k$.
In this case, the action of $I_K$ on
the $\ell$-adic Tate module
$T_\ell A_{\overline{K}}$
of $A$ is unipotent and factors through
$
t_{\ell} \colon I_K \to \Z_\ell(1)
$
for every $\ell \neq p$.
Let $\sigma \in I_K$ be an element such that,
for every $\ell \neq p$,
the image $t_{\ell}(\sigma) \in \Z_\ell(1)$ is a generator.

Since the quotient $\mathscr{A}_s/\mathscr{A}^0_s$ is a finite \'etale group scheme over $k$,
for all but finitely many $\ell \neq p$,
we have
\[
A[\ell^n](\overline{K})^{I_K}=\mathscr{A}^0_s[\ell^n](\overline{k})
\]
for every $n \geq 1$
by the N\'eron mapping property and \cite[Section 7.3, Proposition 3]{BLR}.
It follows that
\[
(T_\ell A_{\overline{K}})^{I_K}\otimes_{\Z_\ell}\F_\ell=(T_\ell A_{\overline{K}}\otimes_{\Z_\ell}\F_\ell)^{I_K}
\]
for all but finitely many $\ell \neq p$.
For such $\ell \neq p$, the cokernel of
$\sigma-1$ acting on
$T_\ell A_{\overline{K}}$
is torsion-free by
Lemma \ref{Lemma:torsion-freeness of cohomology groups}.
Note that we have $(\sigma-1)^2=0$ on $T_\ell A_{\overline{K}}$.
Therefore
we see that Conjecture \ref{Conjecture:A torsion analogue of the weight-monodromy conjecture} for
$H^1_{\et}(A_{\overline{K}}, \F_\ell)$
is true by Theorem \ref{Theorem:Weight-monodromy conjecture} and Proposition \ref{Proposition:cokernel of monodromy operator}.

Let $w$ be an integer.
We can define the monodromy filtration on
$\otimes^{w}H^1_{\et}(A_{\overline{K}}, \F_\ell)$
for all but finitely many $\ell \neq p$; see Section \ref{Subsection:A torsion analogue of the weight-monodromy conjecture}.
The assertion of Conjecture \ref{Conjecture:A torsion analogue of the weight-monodromy conjecture} also holds for
$\otimes^{w}H^1_{\et}(A_{\overline{K}}, \F_\ell)$
by the formula in \cite[Proposition (1.6.9)(i)]{WeilII}.
(Although the base field is of characteristic $0$
in \textit{loc.\ cit.}, the same formula holds with $\F_\ell$-coefficients for all but finitely many $\ell \neq p$.)
Since
$H^w_{\et}(A_{\overline{K}}, \F_\ell)\cong\wedge^{w}H^1_{\et}(A_{\overline{K}}, \F_\ell)$
is a direct summand of $\otimes^{w}H^1_{\et}(A_{\overline{K}}, \F_\ell)$ for all but finitely many $\ell \neq p$,
it follows that
Conjecture \ref{Conjecture:A torsion analogue of the weight-monodromy conjecture} holds for
$H^w_{\et}(A_{\overline{K}}, \F_\ell)$.

The proof of Theorem \ref{Theorem:A torsion analogue of the weight-monodromy conjecture} in the case (\ref{abelian variety}) is complete.

\section{The cases of surfaces}\label{Section:Proof of the main theorem:surfaces}

\subsection{Proof of Theorem \ref{Theorem:A torsion analogue of the weight-monodromy conjecture} in the case (\ref{w=2})}\label{Subsection:The case w=2}

We shall prove Theorem \ref{Theorem:A torsion analogue of the weight-monodromy conjecture} in the case (\ref{w=2}).
We retain the notation of Section \ref{Section:A torsion analogue of the weight-monodromy conjecture}.

By Poincar\'e duality, it is enough to prove the case when $w \leq 2$.
By the theory of Picard varieties, the case $w=1$ follows from the case (\ref{abelian variety}).
We may assume $w=2$.
Since we have already proved Theorem \ref{Theorem:A torsion analogue of the weight-monodromy conjecture} in the case (\ref{equal}),
we may assume that the characteristic of $K$ is $0$.
By de Jong's alteration, we may assume that $X$ is connected and projective over $K$; see \cite[Theorem 4.1]{deJong:Alteration}.

Since the hard Lefschetz theorem with $\Q$-coefficients holds for singular cohomology of projective smooth varieties over $\C$,
the hard Lefschetz theorem with $\Z_\ell$-coefficients holds for \'etale cohomology of projective smooth varieties over $K$ for all but finitely many $\ell$.
(See also Remark \ref{Remark:Hard Lefschetz}.)
Therefore we may assume $\dim X=2$.

We may assume further that
there exists a proper strictly semi-stable scheme $\mathfrak{X}$ over $\O_K$ purely of relative dimension $2$
with generic fiber $X$ by de Jong's alteration; see \cite[Theorem 6.5]{deJong:Alteration}.
By Lemma \ref{Lem:implication} and Lemma \ref{Lemma:equivalent},
it suffices to prove Conjecture \ref{Conjecture:weight-monodromy conjecture spectral sequence} for
$(\mathfrak{X}, w=2, \Lambda_\ell=\Z_\ell)$.
We use the same notation as in Section \ref{Section:Torsion-freeness of the weight spectral sequence}.

We fix an element $\sigma \in I_K$ such that,
for every $\ell \neq p$,
the image $t_{\ell}(\sigma) \in \Z_\ell(1)$ is a generator.
Using the generator $t_{\ell}(\sigma)$,
we identify $\Z_\ell(i)$ with $\Z_\ell$.
We shall prove that the map
$
(1 \otimes t_{\ell}(\sigma))^2 \colon E^{-2, 4}_{2, \Z_\ell} \to E^{2, 0}_{2, \Z_\ell}
$
is an isomorphism for all but finitely many $\ell \neq p$.
This map is identified with the map
\[
\Ker(d^{-2, 4}_1 \colon H^0(Y^{(2)}, \Z_\ell) \to H^2(Y^{(1)}, \Z_\ell)) \to \Coker(d^{1, 0}_1 \colon H^0(Y^{(1)}, \Z_\ell) \to H^0(Y^{(2)}, \Z_\ell))
\]
induced by the identity map on $H^0(Y^{(2)}, \Z_\ell)$.
Here we put $H^i(Y^{(j)}, \Z_\ell):= H^i_\et(Y^{(j)}_{\overline{k}}, \Z_\ell)$ for simplicity.
The map $d^{-2, 4}_1$ is a linear combination of Gysin maps
and the map
$d^{1, 0}_1$
is a linear combination of restriction maps.
Since $\dim Y^{(1)} =1$ and $\dim Y^{(2)} =0$,
each cohomology group is the base change of a finitely generated $\Z$-module
and the above morphism is defined over $\Z$.
These $\Z$-structures are independent of $\ell \neq p$.
Hence $E^{-2, 4}_{2, \Z_\ell}$, $E^{2, 0}_{2, \Z_\ell}$, and the cokernel of the map
$
E^{-2, 4}_{2, \Z_\ell} \to E^{2, 0}_{2, \Z_\ell}
$
are torsion-free for all but finitely many $\ell \neq p$.
Therefore the assertion follows
from the fact that the map
$
E^{-2, 4}_{2, \Q_\ell} \to E^{2, 0}_{2, \Q_\ell}
$
is an isomorphism for every $\ell \neq p$;
see Theorem \ref{Theorem:Weight-monodromy conjecture} and Remark \ref{Remark:weight-monodromy equivalent spectral sequence}.

To prove that the map
$1 \otimes t_{\ell}(\sigma) \colon
E^{-1, 3}_{2, \Z_\ell} \to E^{1, 1}_{2, \Z_\ell}$
is an isomorphism for all but finitely many $\ell \neq p$,
it suffices to prove that the restriction of the canonical pairing on
$H^1(Y^{(1)}, \Z_\ell)$
to the image of the boundary map
\[
d^{0, 1}_1 \colon E^{0, 1}_{1, \Z_\ell}=H^1(Y^{(0)}, \Z_\ell) \to
E^{1, 1}_{1, \Z_\ell}=H^1(Y^{(1)}, \Z_\ell)
\]
is perfect for all but finitely many $\ell \neq p$.
For every $i$,
let
$\Pic^0_{D_i}$
be the Picard variety of $D_i$,
i.e.\ the underlying reduced subscheme of the identity component of the Picard scheme associated with $D_i$.
Similarly, let
$\Pic^0_{D_i \cap D_j}$ be the Picard variety of $D_i \cap D_j$ for every $i < j$.
Since $D_i$ and $D_i \cap D_j$ are proper smooth schemes, the group schemes $\Pic^0_{D_i}$ and $\Pic^0_{D_i \cap D_j}$ are abelian varieties.
The Kummer sequence gives isomorphisms
$
H^1(D_i, \Z_\ell)
\cong T_{\ell}(\Pic^0_{D_i})_{\overline{k}}
$
and
$
H^1(D_i\cap D_j, \Z_\ell)
\cong T_{\ell}(\Pic^0_{D_i\cap D_j})_{\overline{k}}.
$
(Recall that we have fixed the isomorphism
$\Z_\ell(1) \cong \Z_\ell$.)
Under these isomorphisms,
the map $d^{0, 1}_1$ can be identified with
the homomorphism of Tate modules induced by
a linear combination of pull-back maps
\[
\rho \colon \times_i \Pic^0_{D_i} \to \times_{i<j} \Pic^0_{D_i \cap D_j}.
\]
We write $A:=\times_{i<j} \Pic^0_{D_i \cap D_j}$.
Let $B \subset A$ be the image of $\rho$.
By the Poincar\'e complete reducibility theorem,
the image of
$d^{0, 1}_1$
coincides with
$T_{\ell}B_{\overline{k}}$ for all but finitely many $\ell \neq p$.
The canonical pairing on
$H^1(Y^{(1)}, \Z_\ell)$
is equal to the pairing on
$T_{\ell}A_{\overline{k}}$
induced by a principal polarization on $A_{\overline{k}}$.
The restriction of the pairing
on
$T_{\ell}A_{\overline{k}}$
to
$T_{\ell}B_{\overline{k}}$
is induced by a polarization on
$B_{\overline{k}}$,
which is perfect for all but finitely many $\ell \neq p$.
This fact proves our assertion.

The proof of Theorem \ref{Theorem:A torsion analogue of the weight-monodromy conjecture} in the case (\ref{w=2}) is complete.

\begin{rem}\label{Remark:Hard Lefschetz}
In \cite[Compl\'ement 6]{Gabber},
Gabber announced the hard Lefschetz theorem with $\Z_\ell$-coefficients (for all but finitely many $\ell$) for \'etale cohomology of projective smooth varieties in positive characteristic.
\end{rem}

\section{The cases of varieties uniformized by Drinfeld upper half spaces}\label{Section:Proof of the main theorem: Drinfeld}

\subsection{The $\ell$-independence of the weight-monodromy conjecture in certain cases}\label{Subsection:l-independence of the weight-monodromy conjecture}

In this subsection, we make some preparations for the proof of
Theorem \ref{Theorem:A torsion analogue of the weight-monodromy conjecture} in the case (\ref{Drinfeld}).
Let $k$ be a finite field of characteristic $p$.
Let $Y$ be a projective smooth scheme over $k$.
Let $\ell \neq p$ be a prime number. 
The cycle map for codimension $w$ cycles is denoted by
\[
\cl^w_\ell \colon Z^{w}(Y) \to H^{2w}_{\et}(Y_{\overline{k}}, \Z_\ell(w)),
\]
where $Z^{w}(Y)$ is the group of algebraic cycles of codimension $w$ on $Y$.
We denote by
$N^w(Y):= Z^w(Y)/ \sim_{\mathrm{num}}$
the group of algebraic cycles of codimension $w$ on $Y$ modulo numerical equivalence.
It is known that $N^w(Y)$ is a finitely generated $\Z$-module \cite[Expos\'e XIII, Proposition 5.2]{SGA 6}.

\begin{ass}[Assumption $(\ast)$]\label{Assumption:generated by algebraic cycles}
We say that $Y$ \textit{satisfies the assumption $(\ast)$}
if, for every $\ell \neq p$,
we have $H^{w}_\et(Y_{\overline{k}}, \Q_\ell) =0$ for every odd integer $w$ and the $\Q_\ell$-vector space
$H^{2w}_\et(Y_{\overline{k}}, \Q_\ell(w))$ is spanned by the image of $\cl^w_\ell$ for every $w \geq 0$.
\end{ass}

\begin{lem}\label{Lemma:cycle numerical equivalence}
Let $Y$ be a projective smooth scheme over $k$.
Assume that $Y$ satisfies the assumption $(\ast)$.
\begin{enumerate}
\renewcommand{\labelenumi}{(\roman{enumi})}
    \item The cycle map $\cl^w_\ell$ induces an isomorphism
\[
N^w(Y)\otimes_\Z \Q_{\ell} \cong H^{2w}_\et(Y_{\overline{k}}, \Q_{\ell}(w))
\]
for every $\ell \neq p$ and $w \geq 0$.
    \item For all but finitely many $\ell \neq p$, we have
    $H^{w}_\et(Y_{\overline{k}}, \Z_{\ell})=0$ for every odd integer $w$ and
    the cycle map $\cl^w_\ell$ induces an isomorphism
\[
N^w(Y)\otimes_\Z \Z_{\ell} \cong H^{2w}_\et(Y_{\overline{k}}, \Z_{\ell}(w))
\]
for every $w \geq 0$.
\end{enumerate}
\end{lem}
\begin{proof}
The assertions can be proved by using the same argument as in \cite[Lemma 2.1]{Ito-p-adic uniformized} together with Theorem \ref{Theorem:Gabber}.
\end{proof}

Let $K$ be a non-archimedean local field with residue field $k$.
Let $\mathfrak{X}$ be a projective strictly semi-stable scheme
over $\O_K$ purely of relative dimension $d$.
We use the same notation as in Section \ref{Section:Torsion-freeness of the weight spectral sequence}.
So $D_1, \dotsc, D_m$ are the irreducible components of the special fiber $Y$ of $\mathfrak{X}$
and for every non-empty subset
$I \subset \{1, \dotsc, m \}$,
we define $D_I:=  \cap_{i \in I }D_{i}$.
We will consider the weight spectral sequences arising from $\mathfrak{X}$.
We fix an element $\sigma \in I_K$ such that,
for every $\ell \neq p$,
the image $t_{\ell}(\sigma) \in \Z_\ell(1)$ is a generator.

\begin{prop}\label{Proposition:independence weight-monodromy conjecture}
Let the notation be as above.
Assume that for every non-empty subset
$I \subset \{1, \dotsc, m \}$,
the intersection $D_I$ satisfies the assumption $(\ast)$.
We assume further that,
for some prime number $\ell' \neq p$,
the map
\[
(1 \otimes t_{\ell'}(\sigma))^i \colon E^{-i, w+i}_{2, \Q_{\ell'}} \to E^{i, w-i}_{2, \Q_{\ell'}}
\]
is an isomorphism for all $w, i \geq 0$.
Then Conjecture \ref{Conjecture:weight-monodromy conjecture spectral sequence} for
$\mathfrak{X}$
is true.
\end{prop}

\begin{proof}
Using the generator $t_{\ell}(\sigma)$,
we identify $\Z_\ell(i)$ with $\Z_\ell$.
Let $\Lambda_\ell$ be $\Q_\ell$ (resp.\ $\Z_\ell$).
The map
$d^{v, w}_1 \colon E^{v, w}_{1, \Lambda_{\ell}} \to E^{v+1, w}_{1, \Lambda_{\ell}}$
is a linear combination of Gysin maps and restriction maps, whose coefficients are in $\Z$ and independent of $\ell \neq p$; see \cite[Proposition 2.10]{Saito}.
By Lemma \ref{Lemma:cycle numerical equivalence},
for every $\ell \neq p$
(resp.\ all but finitely many $\ell \neq p$),
this map
is the base change of
a homomorphism of finitely generated $\Z$-modules which is
independent of $\ell \neq p$.
Moreover, the same holds for the map
$
(1 \otimes t_{\ell}(\sigma))^i \colon E^{-i, w+i}_{2, \Lambda_{\ell}} \to E^{i, w-i}_{2, \Lambda_{\ell}}.
$
Conjecture \ref{Conjecture:weight-monodromy conjecture spectral sequence} for
$\mathfrak{X}$ follows from this fact.
\end{proof}

\subsection{Proof of Theorem \ref{Theorem:A torsion analogue of the weight-monodromy conjecture} in the case (\ref{Drinfeld})}\label{Subsection:Proof of theorem product Drinfeld}

We shall explain the precise statement.
Let $K$ be a non-archimedean local field
of characteristic $0$ with residue field $k$.
Let $\Omega^{d}_K$ be the Drinfeld upper half space over $K$ of dimension $d$.
It is a rigid analytic variety over $K$.
Let
$
\Gamma \subset \PGL_{d+1}(K)
$
be a discrete cocompact torsion-free subgroup.
It is known that the quotient
$
\Gamma \backslash \Omega^{d}_K
$
is the rigid analytic variety associated with a projective smooth scheme $X$ over $K$.
In this case,
we say that $X$ is uniformized by a Drinfeld upper half space.
We shall prove Conjecture \ref{Conjecture:A torsion analogue of the weight-monodromy conjecture}
for $X$.

Let $\widehat{\Omega}^{d}_{K}$ be the formal model of $\Omega^{d}_K$ considered in \cite{Mustafin},
which is a flat formal scheme locally of finite type over $\Spf \O_K$.
We can take the quotient
$
\Gamma \backslash \widehat{\Omega}^{d}_{K}.
$
There is a flat projective scheme
$\mathfrak{X}$
over $\Spec \O_K$
whose $\varpi$-adic completion is isomorphic to $\Gamma \backslash \widehat{\Omega}^{d}_{K}$.
Here $\varpi$ is a uniformizer of $K$.
The generic fiber of $\mathfrak{X}$ is isomorphic to $X$.
Let
$D_1, D_2, \dotsc, D_m$
be the irreducible components of the special fiber of $\mathfrak{X}$.
As in the proof of \cite[Theorem 1.1]{Ito-p-adic uniformized},
after replacing $\Gamma$ by its finite index subgroup, we may assume that
$\mathfrak{X}$
is a projective strictly semi-stable scheme over $\O_K$ purely of relative dimension $d$
and,
for every non-empty subset
$I \subset \{ 1, 2, \dotsc, m \}$,
the intersection $D_I:= \cap_{i \in I}D_{i}$ satisfies
the assumption $(\ast)$.
Since the weight-monodromy conjecture for $X$ is
true,
we see that Conjecture \ref{Conjecture:A torsion analogue of the weight-monodromy conjecture} for $X$ is true 
by Lemma \ref{Lemma:equivalent}
and
Proposition \ref{Proposition:independence weight-monodromy conjecture}.

\section{Applications to Brauer groups and Chow groups of codimension two cycles}\label{Section:Some applications}

In this section, let $K$ be a non-archimedean local field with residue field $k$.
Let $p >0$ be the characteristic of $k$.
Let $\chara(F)$ denote the characteristic of a field $F$.

\subsection{Brauer groups}\label{Subsectoin:Brauer groups}

First, we recall well known results on the Chern class maps for divisors.
Let $Z$ be a proper smooth scheme over a field $F$.
Let $\NS(Z_{\overline{F}})$ be the N\'eron-Severi group of $
Z_{\overline{F}}$, which is a finitely generated $\Z$-module.
The absolute Galois group $G_F$ of $F$
acts on $\NS(Z_{\overline{F}})$
via the isomorphism
$\Aut(\overline{F}/F) \cong G_F=\Gal(F^{\sep}/F)$.
Let $\Lambda_\ell$ be either $\Q_\ell$ or $\Z_\ell$.
We put
$\NS(Z_{\overline{F}})_{\Lambda_\ell}:=\NS(Z_{\overline{F}})\otimes_\Z \Lambda_\ell$.
The Chern class map with $\Lambda_\ell$-coefficients gives
an injection
\[
\NS(Z_{\overline{F}})_{\Lambda_\ell} \hookrightarrow H^{2}_\et(Z_{\overline{F}}, \Lambda_\ell(1))
\]
for every $\ell \neq \chara(F)$.

\begin{lem}\label{Lemma:decompositon Neron-Severi group}
Let the notation be as above.
Let $\Lambda_\ell = \Q_\ell$ (resp.\ $\Lambda_\ell = \Z_\ell$).
Then there exists a $\Lambda_\ell$-submodule
$M_\ell \subset H^{2}_\et(Z_{\overline{F}}, \Lambda_\ell(1))$
stable by the action of $G_F$
such that
the injection
$\NS(Z_{\overline{F}})_{\Lambda_\ell}\hookrightarrow H^{2}_\et(Z_{\overline{F}}, \Lambda_\ell(1))$
gives a decomposition
\[
H^{2}_\et(Z_{\overline{F}}, \Lambda_\ell(1)) \cong \NS(Z_{\overline{F}})_{\Lambda_\ell} \oplus M_\ell
\]
as a $G_F$-module for every $\ell \neq \chara(F)$
(resp.\ all but finitely many $\ell \neq \chara(F)$).
\end{lem}

\begin{proof}
We may assume that $Z$ is connected.
We first assume that $Z$ is projective.
Let $d:=\dim Z$.
If $d=1$,
then
$\NS(Z_{\overline{F}})_{\Lambda_\ell} \to H^{2}_\et(Z_{\overline{F}}, \Lambda_\ell(1))$
is an isomorphism
for every $\ell \neq \chara(F)$
and the assertion is trivial.
So we assume that $d \geq 2$.
Let $D$ be an ample divisor on $Z$.
The cohomology class of $D$ in
$H^{2}_\et(Z_{\overline{F}}, \Lambda_\ell(1))$
is also denoted by $D$.
Let
$D^{d-2} \in H^{2d-4}_\et(Z_{\overline{F}}, \Lambda_\ell(d-2))$
be the $(d-2)$-times self-intersection of $D$ with respect to the cup product.
We have the following $G_F$-equivariant map:
\begin{align*}
    f_D \colon H^{2}_\et(Z_{\overline{F}}, \Lambda_\ell(1)) &\to \Hom_{\Lambda_{\ell}}(\NS(Z_{\overline{F}})_{\Lambda_\ell}, \Lambda_{\ell})\\
    x &\mapsto (y \mapsto \mathrm{tr}(D^{d-2} \cup x \cup y)),
\end{align*}
where $D^{d-2} \cup x \cup y \in H^{2d}_\et(Z_{\overline{F}}, \Lambda_\ell(d))$ is the cup product of the triple
$(D^{d-2}, x, y)$,
and
$\mathrm{tr} \colon H^{2d}_\et(Z_{\overline{F}}, \Lambda_\ell(d)) \to \Lambda_\ell$ is the trace map.
For every $\ell \neq \chara(F)$
(resp.\ all but finitely many $\ell \neq \chara(F)$),
the restriction of the map $f_D$ to $\NS(Z_{\overline{F}})_{\Lambda_\ell}$
is an isomorphism, and hence $f_D$ gives a
$G_F$-equivariant splitting of
$
\NS(Z_{\overline{F}})_{\Lambda_\ell} \hookrightarrow H^{2}_\et(Z_{\overline{F}}, \Lambda_\ell(1)).
$
This fact proves our claim.

The general case can be reduced to the case where $Z$ is projective as follows.
We may assume that $F$ is perfect after replacing $F$ by the perfect closure of it.
By \cite[Theorem 4.1]{deJong:Alteration},
there exists an alteration
$Z' \to Z$
such that $Z'$ is a projective smooth connected scheme over $F$.
Since we have already proved the assertion for $Z'$,
it suffices to prove the claim that the pull-back map
$
\NS(Z_{\overline{F}})_{\Lambda_\ell} \to \NS(Z'_{\overline{F}})_{\Lambda_\ell}
$
gives a decomposition
$
\NS(Z'_{\overline{F}})_{\Lambda_\ell} \cong \NS(Z_{\overline{F}})_{\Lambda_\ell} \oplus N_\ell
$
as a $G_F$-module for every $\ell \neq \chara(F)$
(resp.\ all but finitely many $\ell \neq \chara(F)$).
The pull-back map
$
\NS(Z_{\overline{F}})_{\Q} \to \NS(Z'_{\overline{F}})_{\Q}
$
is a $G_F$-equivariant injection.
Since both $\NS(Z_{\overline{F}})$ and $\NS(Z'_{\overline{F}})$ are finitely generated $\Z$-modules and the action of $G_F$ on $\NS(Z'_{\overline{F}})$ factors through a finite quotient of $G_F$, the claim follows.
\end{proof}

For a scheme $Z$, let
$
\Br(Z):= H^{2}_\et(Z, \G_m)
$
be the cohomological Brauer group.
Recall that
$\Br(Z)$
is a torsion abelian group if $Z$ is a Noetherian regular scheme; see \cite[Corollaire 1.8]{Grothendieck}.
For an integer $n$,
let $\Br(Z)[n]$ be the set of elements killed by $n$.
Let
$\Br(Z)[p']$
be the prime-to-$p$ torsion part, i.e.\ the set of elements
$x \in \Br(Z)$
such that we have $nx=0$ for some non-zero integer $n$ which is not divisible by $p$.

Let $X$ be a proper smooth scheme over the non-archimedean local field $K$.
Let $\ell \neq \chara(K)$ be a prime number.
Let
\[
\ch_{\Q_\ell} \colon \Pic(X)_{\Q_\ell}:=\Pic(X)\otimes_{\Z}{\Q_\ell} \to H^{2}_\et(X_{\overline{K}}, \Q_\ell(1))
\]
be the $\ell$-adic Chern class map
and let
\[
\ch_{\F_\ell} \colon \Pic(X) \to H^{2}_\et(X_{\overline{K}}, \F_\ell(1))
\]
be the $\ell$-torsion Chern class map.
They induce homomorphisms
\[
\widetilde{\ch}_{\Q_\ell} \colon \Pic(X)_{\Q_\ell} \to H^{2}_\et(X_{\overline{K}}, \Q_\ell(1))^{G_K}
\quad
\text{and}
\quad
\widetilde{\ch}_{\F_\ell} \colon \Pic(X) \to H^{2}_\et(X_{\overline{K}}, \F_\ell(1))^{G_K}.
\]
We will also call $\widetilde{\ch}_{\Q_\ell}$ (resp.\ $\widetilde{\ch}_{\F_\ell}$)
the $\ell$-adic (resp.\ $\ell$-torsion) Chern class map.
We shall study the relation between the Chern class maps and the $G_K$-fixed part of the cohomological Brauer group $\Br(X_{\overline{K}})$ of $X_{\overline{K}}$.
(Here $G_K$ acts on $\Br(X_{\overline{K}})$ via
$\Aut(\overline{K}/K) \cong G_K$.)

\begin{thm}\label{Theorem:vanishing of Brauer group}
Let $X$ be a proper smooth scheme over $K$.
Assume that the $\ell$-adic Chern class map
$
\widetilde{\ch}_{\Q_\ell}
$
is surjective for all but finitely many $\ell \neq p$.
Then the following assertions hold:
\begin{enumerate}
\renewcommand{\labelenumi}{(\roman{enumi})}
    \item The $\ell$-torsion Chern class map
    $
    \widetilde{\ch}_{\F_\ell}
    $
    is surjective for all but finitely many $\ell \neq p$.
    \item 
    The $G_K$-fixed part $\Br(X_{\overline{K}})[\ell]^{G_K}$ is zero for all but finitely many $\ell \neq p$.
\end{enumerate}
\end{thm}

\begin{proof}
(i) By Lemma \ref{Lemma:decompositon Neron-Severi group},
there is a decomposition
\[
H^{2}_\et(X_{\overline{K}}, \Z_\ell(1)) \cong \NS(X_{\overline{K}})_{\Z_\ell} \oplus M_{\ell}
\]
as a $G_K$-module for all but finitely many $\ell \neq p$.
By the assumption,
we have $M_\ell[1/\ell]^{G_K}=0$ for all but finitely many $\ell \neq p$.
It follows that,
for all but finitely many $\ell \neq p$,
every eigenvalue of a lift $\Frob \in G_K$ of the geometric Frobenius element acting on
$M_\ell[1/\ell]^{I_K}$ is different from $1$.

By
Proposition \ref{Proposition:polynomial Frob vanishing} (i),
there exists a non-zero monic polynomial
$P(T) \in \Z[1/p][T]$
such that,
for all but finitely many $\ell \neq p$,
we have $P(\Frob)=0$
on
$H^{2}_\et(X_{\overline{K}}, \Z_\ell(1))$.
We write $P(T)$ in the form
$(T-1)^m Q(T)$ for some non-negative integer $m$ and $Q(T) \in \Z[1/p][T]$ with $Q(1) \neq 0$.
Then $Q(\Frob)=0$ on $M_\ell[1/\ell]^{I_K}$, and hence $Q(\Frob)=0$ on $M_\ell^{I_K}$ for all but finitely many $\ell \neq p$.
By Corollary \ref{Corollary:cokernel of monodromy operator}, we have $Q(\Frob)=0$ on
$(M_\ell \otimes_{\Z_\ell}\F_\ell)^{I_K}=0$
for all but finitely many $\ell \neq p$.
Since $Q(T)$ and $T-1$ are relatively prime in $\Q[T]$, we have
$(M_\ell \otimes_{\Z_\ell}\F_\ell)^{G_K}=0$
for all but finitely many
$\ell \neq p$
by Lemma \ref{Lemma:relatively prime}.
Now, the assertion follows from the fact that
the natural map
$
\Pic(X) \to (\NS(X_{\overline{K}})\otimes_{\Z} \F_\ell)^{G_K}
$
is surjective for all but finitely many
$\ell \neq p$.

(ii) The Kummer sequence gives a short exact sequence
\[
0 \to  \NS(X_{\overline{K}})\otimes_{\Z} \F_\ell \to H^{2}_\et(X_{\overline{K}}, \F_\ell(1)) \to \Br(X_{\overline{K}})[\ell] \to 0
\]
for every $\ell \neq \chara(K)$.
Thus, by Lemma \ref{Lemma:decompositon Neron-Severi group},
there is a decomposition
\[
H^2_\et(X_{\overline{K}}, \F_\ell(1)) \cong (\NS(X_{\overline{K}})\otimes_{\Z} \F_\ell) \oplus \Br(X_{\overline{K}})[\ell]
\]
as a $G_K$-module for all but finitely many $\ell \neq p$.
Thus the assertion follows from (i).
\end{proof}

\begin{cor}\label{Corollary:finiteness Brauer group}
Assume that $\chara(K)=0$ (resp.\ $\chara(K)=p$).
Let $X$ be a proper smooth scheme over $K$.
Assume that the $\ell$-adic Chern class map
$
\widetilde{\ch}_{\Q_\ell}
$
is surjective for every $\ell \neq \chara(K)$.
Then $\Br(X_{\overline{K}})^{G_K}$
(resp.\ $\Br(X_{\overline{K}})[p']^{G_K}$) is finite.
\end{cor}

\begin{proof}
This follows from Theorem \ref{Theorem:vanishing of Brauer group} and the fact that the union
$\cup_n \Br(X_{\overline{K}})[\ell^{n}]^{G_K}$
is finite for every $\ell \neq \chara(K)$ under the assumptions; see the proof of \cite[Corollary 1.5]{CHT}.
We shall give a proof of this fact for the convenience of the reader.

We put
\[
T_\ell \Br(X_{\overline{K}}):=\plim[n]\Br(X_{\overline{K}})[\ell^n]
\]
and
$
V_\ell \Br(X_{\overline{K}}):= T_\ell \Br(X_{\overline{K}}) \otimes_{\Z_\ell} \Q_\ell.
$
As in the proof of Theorem \ref{Theorem:vanishing of Brauer group}, the Kummer sequence and Lemma \ref{Lemma:decompositon Neron-Severi group} give a decomposition
\[
H^{2}_\et(X_{\overline{K}}, \Q_\ell(1)) \cong \NS(X_{\overline{K}})_{\Q_\ell} \oplus V_\ell \Br(X_{\overline{K}})
\]
as a $G_K$-module for every $\ell \neq \chara(K)$.
By the assumption, we have
$(V_\ell \Br(X_{\overline{K}}))^{G_K}=0$.
Since $T_\ell \Br(X_{\overline{K}})$ is torsion-free,
we have
$(T_\ell \Br(X_{\overline{K}}))^{G_K}=0$
for every $\ell \neq \chara(K)$.
It follows that
$\cup_n \Br(X_{\overline{K}})[\ell^{n}]^{G_K}$
is finite for every $\ell \neq \chara(K)$.
\end{proof}

Here we give an example of a projective smooth scheme over $K$ for which
$
\widetilde{\ch}_{\Q_\ell}
$
is surjective for every $\ell \neq \chara(K)$.

\begin{cor}\label{Corollary:Drinfeld upper half plane}
Let $X$ be a projective smooth scheme over $K$ which is uniformized by a Drinfeld upper half space.
\begin{enumerate}
\renewcommand{\labelenumi}{(\roman{enumi})}
    \item The $\ell$-adic Chern class map
$
\widetilde{\ch}_{\Q_\ell}
$
is surjective for every $\ell \neq \chara(K)$.
    \item The $G_K$-fixed part
$\Br(X_{\overline{K}})^{G_K}$
(resp.\ $\Br(X_{\overline{K}})[p']^{G_K}$) is finite
if $\chara(K)=0$ (resp.\ $\chara(K)=p$).
\end{enumerate}
\end{cor}
\begin{proof}
(i) If $d:=\dim X \neq 2$, then
$H^{2}_\et(X_{\overline{K}}, \Q_{\ell}(1))$ is one-dimensional for every $\ell \neq \chara(K)$.
If $d=2$, then
the $G_{K}$-fixed part
$H^{2}_\et(X_{\overline{K}}, \Q_{\ell}(1))^{G_{K}}$
is one-dimensional for every $\ell \neq \chara(K)$ by \cite[Lemma 7.1]{Ito-p-adic uniformized}.
Therefore,
for any $d \geq 1$,
the $\ell$-adic Chern class map
$
\widetilde{\ch}_{\Q_\ell}
$
is surjective for every $\ell \neq \chara(K)$.

(ii) The assertion follows from (i) and Corollary \ref{Corollary:finiteness Brauer group}.
\end{proof}

\begin{rem}\label{Remark:CHT brauer Tate}
Let $F$ be a field which is finitely generated over its prime subfield.
\begin{enumerate}
\renewcommand{\labelenumi}{(\roman{enumi})}
    \item Assume that $\chara(F)=p>0$.
    Let $Z$ be a projective smooth variety over $F$.
    Cadoret-Hui-Tamagawa proved that the Tate conjecture for divisors on $Z$ implies the finiteness of $\Br(Z_{\overline{F}})[p']^{G_F}$; see \cite[Corollary 1.5]{CHT}.
    (If $F$ is finite, this result was proved by Tate; see also the references given in \cite[Section 4]{Tate2}.)
    \item Assume that $\chara(F)=0$.
    Let $Z$ be an abelian variety or a K3 surface over $F$.
    By using the Tate conjecture for divisors on $Z$ and its torsion analogue,
    Skorobogatov-Zarhin proved that 
    $\Br(Z_{\overline{F}})^{G_F}$ is finite; see \cite{Skorobogatov-Zarhin2} for details.
\end{enumerate}
\end{rem}

\begin{rem}\label{Remark:surjectivity l-independence}
Let $X$ be a proper smooth scheme over $K$.
Assume that $\chara(K)=p$ or $\dim X=2$.
If
the $\ell'$-adic Chern class map
$
\widetilde{\ch}_{\Q_{\ell'}}
$
is surjective for some $\ell' \neq p$,
then the same holds for every prime number $\ell \neq \chara(K)$.
For $\ell \neq p$,
this fact can be proved by using Lemma \ref{Lemma:decompositon Neron-Severi group}
and
the $\ell$-independence conjecture stated in
Remark \ref{Remark:l-independence of characteristic poly}
(it is a theorem under the assumptions).
If $\chara(K)=0$, $\dim X =2$, and $\ell=p$,
we use a $p$-adic analogue of the $\ell$-independence conjecture for
the Weil-Deligne representation associated with $H^{2}_\et(X_{\overline{K}}, \Q_{p})$; see \cite[Theorem 3.1]{Ochiai}.
\end{rem}

\subsection{Chow groups of codimension two cycles}\label{Subsection:Chow groups}

In this subsection,
following the strategy of Colliot-Th\'el\`ene and Raskind \cite{CR},
we show some finiteness properties of the Chow group of codimension two cycles on a proper smooth scheme over $K$.

First, we briefly recall a $p$-adic analogue of the weight-monodromy conjecture.
Assume that $\chara(K)=0$.
Let $W_K$ be the Weil group of $K$.
Let $X$ be a proper smooth scheme over $K$.
Let
\[
\WD(H^{w}_\et(X_{\overline{K}}, \overline{\Q}_{p}))
\]
be the Weil-Deligne representation of $W_K$ over $\overline{\Q}_{p}$ associated with $H^{w}_\et(X_{\overline{K}}, \overline{\Q}_{p})$; see \cite[p.469]{TY}.
We say that
the $p$-adic analogue of the weight-monodromy conjecture holds for $(X, w)$
if
$\WD(H^{w}_\et(X_{\overline{K}}, \overline{\Q}_{p}))$
is pure of weight $w$ in the sense of \cite[p.471]{TY}.

Assume that there exists a proper strictly semi-stable scheme $\mathfrak{X}$ over $\O_K$
purely of relative dimension $d$
whose generic fiber is isomorphic to $X$.
Let $Y$ be the special fiber of $\mathfrak{X}$.
Then,
by
the semi-stable comparison isomorphism
\cite[Theorem 0.2]{Tsuji},
the $p$-adic analogue of the weight-monodromy conjecture holds for $(X, w)$ if and only if
the assertion of \cite[Conjecture 3.27]{Mokrane} holds for
the logarithmic crystalline cohomology group
$H^w_{\log \cris}(Y/W(k))[1/p]$,
where we endow $Y$ with the canonical log structure arising from the strictly semi-stable scheme $\mathfrak{X}$.
(Here $W(k)$ is the ring of Witt vectors of $k$.)

The following results are analogues of \cite[Theorem 1.5 and Theorem 1.5.1]{CR}.

\begin{prop}\label{Proposition:vanishing etale cohomology}
Let $X$ be a proper smooth scheme over $K$ and $w$ an integer.
Let $i$ be an integer with $w<2i$.
\begin{enumerate}
\renewcommand{\labelenumi}{(\roman{enumi})}
    \item Assume that Conjecture \ref{Conjecture:Weight-monodromy} holds for $(X, w)$.
    Then $H^{w}_\et(X_{\overline{K}}, \Q_{\ell}/\Z_\ell(i))^{G_K}$ is finite for every $\ell \neq p$.
    Assume further that Conjecture \ref{Conjecture:A torsion analogue of the weight-monodromy conjecture} for $(X, w)$ is true.
    Then we have
    $H^{w}_\et(X_{\overline{K}}, \Q_{\ell}/\Z_\ell(i))^{G_K}=0$
    for all but finitely many $\ell \neq p$.
    \item If $\chara{(K)}=0$ and the $p$-adic analogue of the weight-monodromy conjecture is true for $(X, w)$,
    then $H^{w}_\et(X_{\overline{K}}, \Q_p/\Z_p(i))^{G_K}$ is finite.
\end{enumerate}
\end{prop}

\begin{proof}
For every $\ell \neq \chara(K)$,
we have the following exact sequence of $G_K$-modules:
\[
H^{w}_\et(X_{\overline{K}}, \Z_\ell(i)) \to H^{w}_\et(X_{\overline{K}}, \Q_\ell(i)) \overset{f_\ell}{\to} H^{w}_\et(X_{\overline{K}}, \Q_\ell/\Z_\ell(i)) \to H^{w+1}_\et(X_{\overline{K}}, \Z_\ell(i))_{\mathrm{tor}} \to 0.
\]
Here
$H^{w+1}_\et(X_{\overline{K}}, \Z_\ell(i))_{\mathrm{tor}}$ is the torsion part of
$H^{w+1}_\et(X_{\overline{K}}, \Z_\ell(i))$.
Let
$H_\ell$ denote the free part of $H^{w}_\et(X_{\overline{K}}, \Z_\ell(i))$.
We will use
the continuous cohomology group
$
H^j(G_K, H_\ell)
$
defined in \cite[Section 2]{Tate}.
It is a finitely generated $\Z_\ell$-module for every $\ell \neq \chara(K)$.

(i) We assume that Conjecture \ref{Conjecture:Weight-monodromy} holds for $(X, w)$.
Since $w < 2i$,
it follows that
$H^{w}_\et(X_{\overline{K}}, \Q_\ell(i))^{G_K}=0$ for every $\ell \neq p$.
To show that $H^{w}_\et(X_{\overline{K}}, \Q_{\ell}/\Z_\ell(i))^{G_K}$ is finite for every $\ell \neq p$,
it suffices to show that
$(\Im f_\ell)^{G_K}$ is finite for every $\ell \neq p$.
For every $\ell \neq p$, since 
$H^{w}_\et(X_{\overline{K}}, \Q_\ell(i))^{G_K}=0$,
we see that
$(\Im f_\ell)^{G_K}$
is isomorphic to the torsion part of
$
H^1(G_K, H_\ell)
$
by \cite[Proposition (2.3)]{Tate}.
Hence $(\Im f_\ell)^{G_K}$ is finite.

Assume further that Conjecture \ref{Conjecture:A torsion analogue of the weight-monodromy conjecture} for $(X, w)$ is true.
Since $H^{w+1}_\et(X_{\overline{K}}, \Z_\ell(i))_{\mathrm{tor}}=0$ for all but finitely many $\ell \neq p$ by Theorem \ref{Theorem:Gabber},
we have
$\Im f_\ell=H^{w}_\et(X_{\overline{K}}, \Q_{\ell}/\Z_\ell(i))$
for all but finitely many $\ell \neq p$.
Thus, 
to show that $H^{w}_\et(X_{\overline{K}}, \Q_{\ell}/\Z_\ell(i))^{G_K}=0$
for all but finitely many $\ell \neq p$,
it suffices to prove that
the $\Z_\ell$-module $H^1(G_K, H_\ell)$ is torsion-free for all but finitely many $\ell \neq p$.
We have the following exact sequence:
\[
0 \to H^1(G_k, H^{I_K}_\ell) \to H^1(G_K, H_\ell) \to H^1(I_K, H_\ell).
\]
Let $\Frob \in G_K$ be a lift of the geometric Frobenius element.
We have
\[
H^1(G_k, H^{I_K}_\ell)=\Coker(\Frob-1 \colon H^{I_K}_\ell \to H^{I_K}_\ell)
\quad
\text{and}
\quad
H^1(I_K, H_\ell)= (H_\ell)_{I_K}\otimes_{\Z_\ell}\Z_\ell(-1).
\]

We have
$H^{I_K}_\ell \otimes_{\Z_\ell} \Q_\ell \subset M_{0, \Q_\ell}\otimes_{\Q_\ell}\Q_\ell(i)$,
where $M_{0, \Q_\ell}$ is the $0$-th part of the monodromy filtration on $H^{w}_\et(X_{\overline{K}}, \Q_\ell)$.
By Proposition \ref{Proposition:polynomial Frob vanishing} (iii),
there exists a non-zero monic polynomial $P(T) \in \Z[1/p][T]$
such that
every root of $P(T)$ has complex absolute values
$q^{(w+j)/2}$ with $j \leq -2i$ and,
for every $\ell \neq p$,
we have
$P(\Frob)=0$ on $H^{I_K}_\ell \otimes_{\Z_\ell} \Q_\ell$.
Thus
we also have
$P(\Frob)=0$
on
$H^{I_K}_\ell$
for every $\ell \neq p$.
Since $w<2i$, the polynomials
$P(T)$ and $T-1$
are relatively prime in $\Q[T]$.
Thus, we have $H^1(G_k, H^{I_K}_\ell)=0$ for all but finitely many $\ell \neq p$ by Lemma \ref{Lemma:relatively prime}.
Now,
it remains to prove that
the $\Z_\ell$-module
$H^1(I_K, H_\ell)$
is torsion-free for all but finitely many $\ell \neq p$.
This follows from Proposition \ref{Proposition:cokernel of monodromy operator}.

(ii) If $\chara{(K)}=0$ and the $p$-adic analogue of the weight-monodromy conjecture holds for $(X, w)$,
then we have
$H^{w}_\et(X_{\overline{K}}, \Q_p(i))^{G_K}=0$ if $w < 2i$.
Then the same argument as above shows that $H^{w}_\et(X_{\overline{K}}, \Q_p/\Z_p(i))^{G_K}$ is finite.

The proof of Proposition \ref{Proposition:vanishing etale cohomology} is complete.
\end{proof}

Let $X$ be a proper smooth scheme over $K$.
The Chow group of codimension two cycles on $X_{\overline{K}}$ is denoted by
$\CH^2(X_{\overline{K}})$.
By combining Proposition \ref{Proposition:vanishing etale cohomology} and \cite[Proposition 3.1]{CR}, we have the following results on the torsion part of $\CH^2(X_{\overline{K}})$,
which are local analogues of \cite[Theorem 3.3 and Theorem 3.4]{CR}.

\begin{cor}\label{Corollary:finiteness of Chow group}
Let $X$ be a proper smooth scheme over $K$.
\begin{enumerate}
\renewcommand{\labelenumi}{(\roman{enumi})}
    \item Assume that Conjecture \ref{Conjecture:Weight-monodromy} and Conjecture \ref{Conjecture:A torsion analogue of the weight-monodromy conjecture} hold for $(X, w=3)$.
The prime-to-$p$ torsion part of $\CH^2(X_{\overline{K}})^{G_K}$ is finite.
    \item Assume that $\chara{(K)}=0$ and the $p$-adic analogue of the weight-monodromy conjecture holds for
    $(X, w=3)$.
    Then $\cup_{n} \CH^2(X_{\overline{K}})[p^n]^{G_K}$ is finite.
\end{enumerate}
\end{cor}
\begin{proof}
By \cite[Proposition 3.1]{CR},
there is a $G_K$-equivariant injection
\[
\cup_{n} \CH^2(X_{\overline{K}})[\ell^{n}] \hookrightarrow H^{3}_\et(X_{\overline{K}}, \Q_{\ell}/\Z_\ell(2))
\]
for every $\ell \neq \chara(K)$.
Thus the assertions follow from Proposition \ref{Proposition:vanishing etale cohomology}.
\end{proof}

\begin{cor}\label{Corollary:finiteness of Chow group unconditional}\
\begin{enumerate}
\renewcommand{\labelenumi}{(\roman{enumi})}
    \item
If $(X, w=3)$ satisfies one of the conditions (\ref{equal})--(\ref{complete intersection}) in Theorem \ref{Theorem:Weight-monodromy conjecture},
then the prime-to-$p$ torsion part of $\CH^2(X_{\overline{K}})^{G_K}$ is finite.
    \item Assume that $\chara{(K)}=0$ and $(X, w=3)$ satisfies one of the conditions (\ref{abelian variety})--(\ref{Drinfeld}) in Theorem \ref{Theorem:Weight-monodromy conjecture}.
    Then the torsion part of $\CH^2(X_{\overline{K}})^{G_K}$ is finite.
\end{enumerate}
\end{cor}

\begin{proof}
(i) Use Theorem \ref{Theorem:Weight-monodromy conjecture}, Theorem \ref{Theorem:A torsion analogue of the weight-monodromy conjecture}, and Corollary \ref{Corollary:finiteness of Chow group} (i).

(ii)  Under the assumptions,
the $p$-adic analogue of the weight-monodromy conjecture holds for $(X, w=3)$.
Indeed, if $X$ is an abelian variety over $K$, then this is well known;
since we have
$H^w_{\et}(X_{\overline{K}}, \Q_p)\cong\wedge^{w}H^1_{\et}(X_{\overline{K}}, \Q_p)$,
it suffices to prove
the $p$-adic analogue of the weight-monodromy conjecture for $(Z, w=1)$ for every proper smooth scheme $Z$ over $K$,
and this follows from
the hard Lefschetz theorem and \cite[Th\'eor\`eme 5.3]{Mokrane}.
If $X$ is a proper smooth surface over $K$,
by Poincar\'e duality, this follows from what we have just seen.
If $X$ is uniformized by a Drinfeld upper half space,
this follows from
\cite[Theorem 6.3]{Ito-p-adic uniformized}.
(See also \cite[3.33]{Mokrane}.)

Therefore, the assertion follows from (i) and Corollary \ref{Corollary:finiteness of Chow group} (ii).
\end{proof}

\begin{rem}\label{Remark:finiteness of chow group, surface}\
\begin{enumerate}
\renewcommand{\labelenumi}{(\roman{enumi})}
    \item If $X$ has good reduction over $\O_K$,
    the finiteness of the prime-to-$p$ torsion part of $\CH^2(X_{\overline{K}})^{G_K}$ was known; see the proof of \cite[Theorem 3.4]{CR}.
    \item We assume that $\chara(K)=0$.
    If $\dim X=2$ or $H^{3}_\et(X_{\overline{K}}, \Q_{\ell})=0$ for some (and hence every) $\ell$,
    then the finiteness of the torsion part of $\CH^2(X_{\overline{K}})^{G_K}$ is known; see \cite[Section 4]{CR2} and the proof of \cite[Theorem 4.1]{Shuji}.
When $\dim X=2$, it is a consequence of Roitman's theorem; see \cite[Remark 3.5]{CR} for details.
\end{enumerate}
\end{rem}

\subsection*{Acknowledgements}
The author would like to thank his adviser Tetsushi Ito for several invaluable discussions.
He also pointed out a mistake in an earlier version of this paper.
The author also would like to thank Teruhisa Koshikawa and Kanetomo Sato for helpful comments.
The work of the author was supported by JSPS Research Fellowships for Young Scientists KAKENHI Grant Number 18J22191.

\end{document}